\documentclass[reqno]{amsart}
\usepackage[margin=1in]{geometry}  
\usepackage{graphicx}
\usepackage{amsmath}
\usepackage{amsfonts}
\usepackage{amsthm,bbm}
\usepackage{amssymb}
\usepackage{amscd}
 \usepackage{listings}
\usepackage{times}
\usepackage{comment}
\usepackage{tikz,ifthen}
\usetikzlibrary{arrows, matrix}
\usetikzlibrary{decorations.pathreplacing,decorations.markings}
\usetikzlibrary{cd}
\usepackage{tikz-cd}
\usepackage{calc}
\usetikzlibrary{calc}
\usepackage{stmaryrd}
\usepackage{xcolor}
\usepackage{mathrsfs}
\usepackage[new]{old-arrows}
\usepackage{mathtools}               

\makeatletter                     
\def\slashedarrowfill@#1#2#3#4#5{$\m@th\thickmuskip0mu\medmuskip\thickmuskip\thinmuskip\thickmuskip 
   \relax#5#1\mkern-7mu
   \cleaders\hbox{$#5\mkern-2mu#2\mkern-2mu$}\hfill 
   \mathclap{#3}\mathclap{#2}
   \cleaders\hbox{$#5\mkern-2mu#2\mkern-2mu$}\hfill
   \mkern-7mu#4$
}               
\def\rightslashedarrowfill@{
\slashedarrowfill@\relbar\relbar\mapstochar\rightarrow} 
\newcommand\xslashedrightarrow[2][]{
  \ext@arrow 0055{\rightslashedarrowfill@}{#1}{#2}}  
\newcommand{\proarrow}{\xslashedrightarrow{}} 
\makeatother  

\theoremstyle{definition}
\newtheorem{thm}{Theorem}[section]
\newtheorem{lem}[thm]{Lemma}
\newtheorem{prop}[thm]{Proposition}
\newtheorem{cor}[thm]{Corollary}
\newtheorem{example}[thm]{Example}
\newtheorem{question}{Question}
\newtheorem{mainthm}{Theorem}

\newtheorem{defprop}[thm]{Definition/Proposition}
\newtheorem{definition}[thm]{Definition}
\newtheorem{rem}[thm]{Remark}

\newcommand{\ie}{i.e.\ }
\newcommand{\id}{\operatorname{id}} 
\newcommand{\Gm}{\mathbf{G}_m}
\newcommand{\kk}{\mathsf{k}}

\usepackage[cmtip,all]{xy}
\newcommand{\longsquiggly}{\xymatrix{{}\ar@{~>}[r]&{}}}

\newcommand{\N}{\mathbf{N}}
\newcommand{\Z}{\mathbf{Z}}
\newcommand{\R}{\mathbf{R}}
\newcommand{\SL}{\mathrm{SL}}
\newcommand{\C}{\mathbf{C}}
\newcommand{\GL}{\mathrm{GL}}
\newcommand{\Cat}{\operatorname{Cat}}
\newcommand{\Fun}{\operatorname{Fun}}
\newcommand{\Hom}{\operatorname{Hom}}
\newcommand{\ihom}{\underline{\Hom}}
\newcommand{\End}{\operatorname{End}}
\newcommand{\iend}{\underline{\End}}
\newcommand{\g}{\mathfrak{g}}
\newcommand{\spec}{\mathrm{Spec}}
\newcommand{\im}{\operatorname{im}}

\newcommand{\Id}{\mathrm{Id}}
\newcommand{\acts}{\curvearrowright}

\newcommand{\Bord}{\operatorname{Bord}}
\newcommand{\tglBord}{\operatorname{Bord}_3^{\mathrm{tgl}}}
\newcommand{\Ann}{\mathrm{Ann}}
\newcommand{\ptorus}{{\mathrm{T}^\ast}}
\newcommand{\pSigma}{{\Sigma^\ast}}
\newcommand{\Rep}{\operatorname{Rep}}
\newcommand{\modu}{\mbox{-}\Mod}
\newcommand{\A}{\mathcal{A}}
\newcommand{\M}{\mathcal{M}}
\newcommand{\AM}{{(\A,\M)}}
\newcommand{\opp}{\mathrm{op}}
\newcommand{\Mod}{\operatorname{Mod}}
\newcommand{\Bimod}{\operatorname{Bimod}}
\newcommand{\prl}{\operatorname{Pr}^\mathrm{L}}
\newcommand{\unit}{\mathbf{1}} 
\newcommand{\ind}[1]{\widehat{#1}}
\newcommand{\Ahat}{\ind{\A}}

\newcommand{\Aut}{\operatorname{Aut}}
\newcommand{\Vect}{\operatorname{Vect}}
\newcommand{\totimes}{\widetilde{\otimes}}
\newcommand{\sk}{\operatorname{Sk}}
\newcommand{\skalg}{\operatorname{SkAlg}}

\newcommand{\DqG}{\mathcal{D}_q(G)}
\newcommand{\Oq}{\mathcal{O}_q}
\newcommand{\Dq}{\mathcal{D}_q}
\newcommand{\detq}{\operatorname{det}_q}
\newcommand{\trq}{\operatorname{tr}_q}
\newcommand{\sA}{\mathscr{A}}
\newcommand{\Ocal}{\mathcal{O}}
\newcommand{\Dcal}{\mathcal{D}}
\newcommand{\ract}{\triangleleft}
\newcommand{\lact}{\triangleright}

\usepackage{hyperref}
\hypersetup{
    colorlinks,
    citecolor=blue,
    filecolor=blue,
    linkcolor=blue,
    urlcolor=blue
}

\tikzset{
  strand/.style={semithick}, 
  redstrand/.style={semithick, color=red}, 
  dstrand/.style={semithick,dash pattern=on 0.7 off 1pt},  
  oriented/.style={postaction={decorate},
                   decoration={markings, mark=at position 0.85 with {\arrow{>}}}} 
}

\newcommand{\undercross}[2]{
  \vcenter{\hbox{\begin{tikzpicture}[scale=0.6]
    \draw[#1,oriented] (-0.6,-0.6) -- (0.6,0.6);
    \draw[white,line width=5pt] (0.6,-0.6) -- (-0.6,0.6); 
    \draw[#2,oriented] (0.6,-0.6) -- (-0.6,0.6);
  \end{tikzpicture}}}
}

\newcommand{\overcross}[2]{
  \vcenter{\hbox{\begin{tikzpicture}[scale=0.6]
    \draw[#2,oriented] (0.6,-0.6) -- (-0.6,0.6); 
    \draw[white,line width=5pt] (-0.6,-0.6) -- (0.6,0.6); 
    \draw[#1,oriented] (-0.6,-0.6) -- (0.6,0.6); 
  \end{tikzpicture}}}
}

\newcommand{\parallelstrands}[2]{
  \vcenter{\hbox{\begin{tikzpicture}[scale=0.6]
    \draw[#1] (-0.4,-0.6) -- (-0.4,0.6);
    \draw[#2] (0.4,-0.6) -- (0.4,0.6);
  \end{tikzpicture}}}
}

\newcommand{\singlestrand}[1]{
  \vcenter{\hbox{\begin{tikzpicture}[scale=0.6]
    \draw[#1] (0,-0.6) -- (0,0.6);
  \end{tikzpicture}}}
}
\newcommand{\loopstrandJ}[1]{
  \vcenter{\hbox{\begin{tikzpicture}[scale=0.6]
    \coordinate (C) at (-0.1,0);
    \coordinate (X) at (-0.1,-0.6);
    \coordinate (Y) at (-0.1,0.6);
    \coordinate (A0) at (0.1,0.1);
    \coordinate (A1) at (0.1,-0.3);
    \coordinate (B) at (0.3,-0.1);
    \draw[#1] (A1) [out=180,in=-80] to (C);
    \draw[#1,oriented] (C) [out=100,in=-90] to (Y);
    \draw[white,line width=4.5pt] (X) [out=90,in=-100] to (C);
    \draw[white,line width=4.5pt] (C) [out=80,in=90] to (B);
    \draw[white,line width=4.5pt] (B) [out=-90,in=0] to (A1);
    \draw[#1] (X) [out=90,in=-100] to (C);
    \draw[#1] (C) [out=80,in=180] to (A0);
    \draw[#1] (A0) [out=0,in=90] to (B);
    \draw[#1] (B) [out=-90,in=0] to (A1);
  \end{tikzpicture}}}
}

\newcommand{\circlestrand}[2]{
  \vcenter{\hbox{\begin{tikzpicture}[scale=0.6]
    \ifthenelse{\equal{#2}{ccw}}{
      \draw[#1,oriented] (0,0) circle (0.5);
    }{
      \draw[#1,oriented] (0.5,0) arc[start angle=0,end angle=-360,radius=0.5];
    }
  \end{tikzpicture}}}
}

\newcommand{\Nvertex}[1]{
  \vcenter{\hbox{\begin{tikzpicture}[scale=0.6]
    \tikzset{lowarrow/.style={semithick, postaction={decorate},
                               decoration={markings, mark=at position 0.25 with {\arrow{>}}}}}

    \draw[strand,lowarrow] (-1.2,-1.2) -- (0,-0.2);
    \draw[strand,lowarrow] (-0.7,-1.2) -- (0,-0.2);

    \node at (0,-1.2) {\scriptsize $\cdots$};

    \draw[strand,lowarrow] (0.7,-1.2) -- (0,-0.2);
    \draw[strand,lowarrow] (1.2,-1.2) -- (0,-0.2);

    \filldraw (0,-0.2) circle (2pt);

    \draw[rdstrand,oriented] (0,-0.2) -- (0,1.0);
  \end{tikzpicture}}}
}

\newcommand{\dualNvertex}[1]{
  \vcenter{\hbox{\begin{tikzpicture}[scale=0.6]
    \tikzset{lowarrow/.style={semithick, postaction={decorate},
                               decoration={markings, mark=at position 0.25 with {\arrow{<}}}}}

    \draw[rdstrand,oriented] (0,-1.0) -- (0,-0.2);

    \filldraw (0,-0.2) circle (2pt);

    \draw[strand,lowarrow] (-1.2,0.8) -- (0,-0.2);
    \draw[strand,lowarrow] (-0.7,0.8) -- (0,-0.2);

    \node at (0,0.8) {\scriptsize $\cdots$};

    \draw[strand,lowarrow] (0.7,0.8) -- (0,-0.2);
    \draw[strand,lowarrow] (1.2,0.8) -- (0,-0.2);
  \end{tikzpicture}}}
}

\newcommand{\NtoNvertex}[1]{
  \vcenter{\hbox{\begin{tikzpicture}[scale=0.6]
    \tikzset{
      lowarrow/.style={semithick,postaction={decorate},
                        decoration={markings, mark=at position 0.25 with {\arrow{>}}}},
      dashedarrow/.style={semithick,postaction={decorate},
                           decoration={markings, mark=at position 0.8 with {\arrow{>}}}},
      outarrow/.style={semithick,postaction={decorate},
                        decoration={markings, mark=at position 0.25 with {\arrow{<}}}}
    }

    \draw[strand,lowarrow] (-1.2,-1.5) -- (0,-0.7);
    \draw[strand,lowarrow] (-0.7,-1.5) -- (0,-0.7);
    
    \node at (0,-1.5) {\scriptsize $\cdots$};

    \draw[strand,lowarrow] (0.7,-1.5) -- (0,-0.7);
    \draw[strand,lowarrow] (1.2,-1.5) -- (0,-0.7);

    \filldraw (0,-0.7) circle (2pt);

    \draw[dstrand, dashedarrow] (0,-0.7) -- (0,0.0);

    \filldraw (0,0.0) circle (2pt);

    \draw[strand,outarrow] (-1.2,0.8) -- (0,0.0);
    \draw[strand,outarrow] (-0.7,0.8) -- (0,0.0);

    \node at (0,0.8) {\scriptsize $\cdots$};

    \draw[strand,outarrow] (0.7,0.8) -- (0,0.0);
    \draw[strand,outarrow] (1.2,0.8) -- (0,0.0);
  \end{tikzpicture}}}
}

\newcommand{\Nbox}[1]{
  \vcenter{\hbox{\begin{tikzpicture}[scale=0.6]
    \tikzset{lowarrow/.style={semithick, postaction=
                               {decorate},
                               decoration={markings, mark=at position 0.6 with {\arrow{>}}}}}
    \draw[strand,lowarrow] (-1.2,-1.5) -- (-1.2,-0.5);
    \draw[strand,lowarrow] (-0.8,-1.5) -- (-0.8,-0.5);
    \node at (0,-1) {\scriptsize $\cdots$};
    \draw[strand,lowarrow] (0.8,-1.5) -- (0.8,-0.5);
    \draw[strand,lowarrow] (1.2,-1.5) -- (1.2,-0.5);
    \draw[thick,fill=white] (-1.4,-0.5) rectangle (1.4,0.5);
    \node at (0,0) {$#1$};
    \draw[strand,lowarrow] (-1.2,0.5) -- (-1.2,1.5);
    \draw[strand,lowarrow] (-0.8,0.5) -- (-0.8,1.5);
    \node at (0,1) {\scriptsize $\cdots$};
    \draw[strand,lowarrow] (0.8,0.5) -- (0.8,1.5);
    \draw[strand,lowarrow] (1.2,0.5) -- (1.2,1.5);
  \end{tikzpicture}}}
}

\newcommand{\brmodstrands}{
\vcenter{\hbox{   \begin{tikzpicture}[scale = 1, baseline={([yshift=-.5ex]current bounding box.center)}]
    \draw[redstrand] (.5,-.6) -- (.5,0);
    \draw[white, line width = 5pt] (0,-.6) .. controls (0,-.4) and (.7,-.3) .. (.7, 0);
    \draw[strand, ->- = .55] (0,-.6) .. controls (0,-.4) and (.7,-.3) .. (.7, 0) .. controls (.7,.3) and (0,.4) .. (0,.6);
    \draw[white, line width=5pt] (.5, 0) -- (.5,.6);
    \draw[redstrand] (.5,0) -- (.5,.6);
    \end{tikzpicture}}}
}

\newcommand{\brmodstrandsunoriented}{
\vcenter{\hbox{   \begin{tikzpicture}[scale = 1, baseline={([yshift=-.5ex]current bounding box.center)}]
    \draw[redstrand] (.5,-.6) -- (.5,0);
    \draw[white, line width = 5pt] (0,-.6) .. controls (0,-.4) and (.7,-.3) .. (.7, 0);
    \draw[strand] (0,-.6) .. controls (0,-.4) and (.7,-.3) .. (.7, 0) .. controls (.7,.3) and (0,.4) .. (0,.6);
    \draw[white, line width=5pt] (.5, 0) -- (.5,.6);
    \draw[redstrand] (.5,0) -- (.5,.6);
    \end{tikzpicture}}}
}

\newcommand{\invbrmodstrands}{
\vcenter{\hbox{   \begin{tikzpicture}[scale = 1, baseline={([yshift=-.5ex]current bounding box.center)}]
    \draw[redstrand] (.5,0) -- (.5,.6);
    \draw[white, line width = 5pt] (.7, 0) .. controls (.7,.3) and (0,.4) .. (0,.6);
    \draw[strand, ->- = .55] (0,-.6) .. controls (0,-.4) and (.7,-.3) .. (.7, 0) .. controls (.7,.3) and (0,.4) .. (0,.6);
    \draw[white, line width = 5pt] (.5,-.6) -- (.5, 0);
    \draw[redstrand] (.5,-.6) -- (.5, 0);
    \end{tikzpicture}}}
}

\newcommand{\invbrmodstrandsunoriented}{
\vcenter{\hbox{   \begin{tikzpicture}[scale = 1, baseline={([yshift=-.5ex]current bounding box.center)}]
    \draw[redstrand] (.5,0) -- (.5,.6);
    \draw[white, line width = 5pt] (.7, 0) .. controls (.7,.3) and (0,.4) .. (0,.6);
    \draw[strand] (0,-.6) .. controls (0,-.4) and (.7,-.3) .. (.7, 0) .. controls (.7,.3) and (0,.4) .. (0,.6);
    \draw[white, line width = 5pt] (.5,-.6) -- (.5, 0);
    \draw[redstrand] (.5,-.6) -- (.5, 0);
    \end{tikzpicture}}}
}

\newcommand{\detbrmodstrands}{
\vcenter{\hbox{   \begin{tikzpicture}[scale = 1, baseline={([yshift=-.5ex]current bounding box.center)}]
    \draw[redstrand] (.5,-.6) -- (.5,0);
    \draw[white, line width = 5pt] (0,-.6) .. controls (0,-.4) and (.7,-.3) .. (.7, 0);
    \draw[dstrand, ->- = .55] (0,-.6) .. controls (0,-.4) and (.7,-.3) .. (.7, 0) .. controls (.7,.3) and (0,.4) .. (0,.6);
    \draw[white, line width=5pt] (.5, 0) -- (.5,.6);
    \draw[redstrand] (.5,0) -- (.5,.6);
    \end{tikzpicture}}}
}

\newcommand{\brmod}[3]{
    \draw[redstrand] ($(#1)+(#2,0)$) -- ($(#1)+(#2,#3)$);
    \draw[white, line width = 5pt] (#1) .. controls ($(#1)+(0,{#3/6})$) and ($(#1)+(#2+0.2,{#3/4})$) .. ($(#1)+(#2+0.2,{#3/2})$);
    \draw[strand, ->- = .55] (#1) .. controls ($(#1)+(0,{#3/6})$) and ($(#1)+(#2+0.2,{#3/4})$) .. ($(#1)+(#2+0.2,{#3/2})$) .. controls ($(#1)+(#2+0.2,{3*#3/4})$) and ($(#1)+(0,{5*#3/6})$) .. ($(#1)+(0,#3)$);
    \draw[white, line width=5pt] ($(#1)+(#2,{#3/2})$) -- ($(#1)+(#2,#3)$);
    \draw[redstrand] ($(#1)+(#2,{#3/2})$) -- ($(#1)+(#2,#3)$);
}

\newcommand{\brmodtrace}{
\vcenter{\hbox{   \begin{tikzpicture}[scale = 1, baseline={([yshift=-.5ex]current bounding box.center)}]
    \draw[redstrand] (.3,-.6) -- (.3,0);
    \draw[white,line width=5pt] (0,0) arc (-180:180:.3);
    \draw[strand,->-=.5] (0,0) arc (-180:180:.3);
    \draw[white, line width=5pt] (.3,0) -- (.3,.6);
    \draw[redstrand] (.3,0) -- (.3,.6);
    \end{tikzpicture}}}
}

\newcommand{\brmodtraceunoriented}{
\vcenter{\hbox{   \begin{tikzpicture}[scale = 1, baseline={([yshift=-.5ex]current bounding box.center)}]
    \draw[redstrand] (.3,-.6) -- (.3,0);
    \draw[white,line width=5pt] (0,0) arc (-180:180:.3);
    \draw[strand] (0,0) arc (-180:180:.3);
    \draw[white, line width=5pt] (.3,0) -- (.3,.6);
    \draw[redstrand] (.3,0) -- (.3,.6);
    \end{tikzpicture}}}
}

\newcommand{\cyl}[2]{
\draw (0,0) ellipse[x radius = #1, y radius = 0.5*#1];
\draw (0,#2) ellipse[x radius = #1, y radius = .5*#1];
\draw (-#1,0) -- (-#1,#2);
\draw (#1,0) -- (#1,#2);
}

\newcommand{\kcrossing}{
  \begin{tikzpicture}[baseline=-0.5ex, scale=0.4]
    \draw[thick] (-0.5,0.5) -- (0.5,-0.5);
    \draw[white, line width=4pt] (-0.5,-0.5) -- (0.5,0.5);
    \draw[thick] (-0.5,-0.5) -- (0.5,0.5);
  \end{tikzpicture}
}
\newcommand{\kvertical}{
  \begin{tikzpicture}[baseline=-0.5ex, scale=0.4]
    \draw[thick] (-0.5,-0.5) to[out=45, in=-45] (-0.5,0.5);
    \draw[thick] (0.5,-0.5) to[out=135, in=-135] (0.5,0.5);
  \end{tikzpicture}
}
\newcommand{\khorizontal}{
  \begin{tikzpicture}[baseline=-0.5ex, scale=0.4]
    \draw[thick] (-0.5,-0.5) to[out=45, in=135] (0.5,-0.5);
    \draw[thick] (-0.5,0.5) to[out=-45, in=-135] (0.5,0.5);
  \end{tikzpicture}
}
\newcommand{\kunknot}{
  \begin{tikzpicture}[baseline=-0.5ex, scale=0.4]
    \draw[thick] (0,0) circle (0.35);
  \end{tikzpicture}
}

\begin{document}

\tikzset{->-/.style={decoration={
  markings,
  mark=at position #1 with {\arrow{>}}},postaction={decorate}}}

\title{Skein theory, line defects, and quantum symmetric pairs}

\author{Eric Yen-Yo Chen, David Jordan, and Iordanis Romaidis}

\begin{abstract}
We construct skein theory for 3-manifolds with embedded line defects, starting from the data of a ribbon tensor category and its balanced braided module category. We focus on line defects arising from quantum symmetric pairs, and prove finiteness properties for their defect skein modules. 
As an application, we consider $\Z_2$-equivariant skein theory and establish an equivalence with defect skein theory in certain settings, leading to a skein theoretical construction of a family of $\mathrm{C}^\vee\mathrm{C}$ DAHA-modules in the Type AIII case. 
\end{abstract}
\maketitle

\setcounter{tocdepth}{1}
\tableofcontents

\section{Introduction}\label{sec:intro}

Skein theory is an elementary approach to study 3-manifolds which emerged from the discovery of Jones' knot invariant and Witten's subsequent interpretation in terms of topological quantum field theory. Despite its simplicity, skein theory has recently found connections to contemporary mathematical topics such as cohomological Donaldson--Thomas theory \cite{Abouzaid-Manolescu, BBDJS, BDIKP, GS, KKPS25,KPS24}, (alpha and beta) factorization homology \cite{AFRbeta,AFT17, Cooke, BrownHaioun, Chun-Yu}, cluster algebras \cite{BrownJ, Bonahon-Wong1,Bonahon-Wong2, Fock-Goncharov06,Fock-Goncharov09,Garoufalidis-Yu, JLSS, Le2019,Mul2016,NeitzkeYan,Panitch-Park}, and mathematical gauge theory \cite{Du Pei,JLanglands}. 

The ubiquity of skein theory is due in part to its rather elementary definition, which we now briefly recall.  We focus on the most well-studied case of the \textit{Kauffman bracket skein module} of an oriented 3-manifold introduced by \cite{Przytycki} and \cite{tur88}, associated to the group $\SL_2$.

\begin{definition} \label{definition KB skein relations} The \textit{Kauffman bracket skein module} of a closed oriented 3-manifold $M$, denoted $\mathrm{Sk}_{\mathrm{KB}}(M)$, is defined as the $\C[q^{\pm1/2}]$-linear span of framed link embeddings in $M$ modulo isotopy and local relations provided by the following two Kauffman bracket relations
\begin{enumerate}
    \item Crossing resolution:
    \begin{equation}\label{eq:KB-sk-relations1}
    \big\langle \kcrossing \big\rangle = q^{1/2} \left\langle \, \kvertical \,  \right\rangle + q^{-1/2} \big\langle \, \khorizontal \, \big \rangle
    \end{equation}

    \item Normalization:
    \begin{equation}\label{eq:KB-sk-relations2}
    \big\langle \,  \kunknot  \, \big\rangle = -(q+q^{-1})
    \end{equation}
\end{enumerate}
\end{definition}

The construction of the Kauffman bracket skein module generalises in two ways, both of which we will require.  Firstly, from an algebraic perspective: the construction of skein modules allows for the algebraic input of a ribbon tensor category $\A$ over commutative ring $R$ which both decorates the skeins in the relevant 3-manifold $M$ and also provides a coherent set of local relations. More precisely, the skein module $\sk_\A(M)$ is defined as the $R$-linear span of embedded $\A$-colored ribbon graphs in $M$, consisting of edges labelled by $\A$-objects and vertices labelled by $\A$-morphisms. Local relations are then encoded by Reshetikhin--Turaev graphical calculus. Our favourite class of examples will be the category of Type I representations of quantum groups, denoted $\Rep_q(G)$, where $G$ is a reductive group, from which the Kauffman bracket skein module may be recovered by setting $\A=\Rep_q(\SL_2)$. 

Secondly, from a topological perspective: the assignment $M\mapsto \sk_\A(M)$ has been extended to surfaces by the definition of the $\A$-skein category $\sk_\A(\Sigma)$, exhibiting skein theory as a 3d (categorified) TQFT \cite{Walker-notes, JohnsonFreyd}, \ie a symmetric monoidal functor 
\[\sk_\A: \Bord_3 \longrightarrow \Bimod\]
from the category of oriented 3-bordisms to the category of bimodules (see Section~\ref{sec:pre}). Such an interpretation brings to bear the powerful formal techniques of factorization homology \cite{AFRbeta,AFT17}, and the communication between 2d/3d has led to significant advances towards the resolution of
Witten's Conjecture on the finite dimensionality of skein modules when $\A = \mathrm{Rep}_q(G)$ \cite{GJS} for generic parameters $q$ and any reductive group $G$, and a full resolution for $G=\SL_2$ in \cite{Belletti-Detcherry,JR}. Interpreting natural algebraic structures that arise, one obtains furthermore conceptual constructions of the finite, the affine, and the double affine Hecke algebra and their representations \cite{MS}, and calculations of skein module dimensions \cite{GJVtori,BJVV}.

The purpose of this paper is twofold: 1) we formalise the construction of skein theory for 3-manifolds with embedded defects of codimension two, and 2) we describe in detail the class of defects arising from quantum symmetric pairs \cite{Noumi,Let02,Let03,BK,Kol14,Kol20}. As an application of this theory, we obtain finiteness results for defect skein modules akin to those for ordinary skein modules. For instance, taking the symmetric pair $(\mathrm{SL}_2, T)$ where $T \subset \mathrm{SL}_2$ is the diagonal torus yields the following defect version of the Kauffman bracket skein module. 
\begin{defprop}
    Let $K \subset M$ be an embedded link in an oriented 3-manifold. The \textit{$(\mathrm{SL}_2,T)$-defect skein module of $(M,K)$} is the $\C[q^{\pm 1/2}]$-linear span of framed link embeddings in $M \setminus K$ modulo isotopy, the Kauffman bracket relations in $M \setminus K$, and the extra defect relations
    \begin{gather} \label{eq:SL2-defect-relation}
 \brmodstrandsunoriented\; + q^{-2} \;  \invbrmodstrandsunoriented \;=\;q^{-1}(t+t^{-1})\;\parallelstrands{strand}{redstrand} \, , \quad \quad \brmodtraceunoriented = (t+t^{-1})  ~  \, \singlestrand{redstrand} 
\end{gather}
    where the red strand represents a local piece of $K$, and $t \in \C[q^{\pm 1/2}]_{(q-1)}$ is a parameter specializing to (a choice of) $\sqrt{-1}$ as $q \to 1$.
\end{defprop}

\begin{mainthm}[Corollary \ref{cor:finiteness-defect-skmod}]\label{thm:def-skmod-finiteness}
Let $K$ be an embedded link in an oriented 3-manifold $M$. Then the $(\mathrm{SL}_2,T)$-defect skein module of $(M,K)$ is finite dimensional when tensored with $\C(q^{1/2})$.
\end{mainthm}
In the case when the defect knot $K$ is the unknot inside the 3-sphere, the resulting defect skein module is 1-dimensional and is populated by \textit{dichromatic link polynomials}, previously studied by \cite{Hoste-Kidwell, Hoste-Przytycki, Lambropoulou}. 

In the rest of the introduction, we discuss briefly the formalism of codimension two defects via braided modules, relevant aspects of the theory of quantum symmetric pairs, and situate our work in relation to recent developments in the relative Langlands program \cite{BZSV} and shifted geometric quantization \cite{Safronov1}. 

\subsection{Codimension two defects from braided module categories}\label{subsec:intro-defects-from-brmod}

From the point of view of knot theory, it is natural to search for an extension of skein theory to 3-manifolds codimension 2 defects, i.e., closed oriented 3-manifolds with embedded links.\footnote{In 3 dimensions, it is typical to refer to codimension 2 defects as \textit{line defects}.} Intuitively speaking, the defect skein module assigned to a 3-manifold $M$ with codimension 2 defect $K \subset M$ is the universal recipient of polynomial knot invariants in $M \setminus K$ which may braid along the fixed link $K$. Alternatively, we may excise out a normal neighborhood of $K$ from $M$ and regard the defect skein module as an invariant of 3-manifolds with boundary, where each boundary component is a genus 1 surface.

In 2d, one considers analogously oriented surfaces with marked points, and the analogous question of assigning a \textit{skein category with defects} leads \cite{BZBJ18b} to a proposed answer which we briefly recall. By observing the equivalence due to \cite{Cooke}
\[\sk_\A(\Sigma)\simeq \int_{\Sigma} \A\] between the skein category of a surface $\Sigma$ and the factorization homology of $\A$ on $\Sigma$, one observes in line with \cite{BZBJ18b} that codimension two defects should be labelled by balanced braided module categories over $\A$. Given such a braided module category $\M$ over $\A$, we construct defect skein theories in 2d/3d with $\M$-labelled line defects, which assigns defect skein modules to 3-manifolds with defect links. The construction parallels the non-defect case by establishing local relations near the line defect using graphical calculus of $\M$, while imposing the usual skein relations determined by $\A$ in the bulk. Our first main result is the following theorem establishing skein theory with line defects as a 3d categorified TQFT: 
\begin{mainthm}[see Sect. \ref{subsec:sk-def} and Thm.\ \ref{thm:def-sk-TQFT}]\label{thm:intro-def-sk-TQFT}
Let $\A$ be a ribbon tensor category and let $\M$ be a balanced braided module category over $\A$. Skein theory with line defects labelled by the pair $(\A,\M)$ defines a symmetric monoidal functor 
\[\sk_{\AM}: \Bord_3^{\mathrm{tgl}}\longrightarrow \Bimod~\]
from the category of oriented bordisms with codimension two defects to the category of bimodules. 
\end{mainthm}
Unpacking slightly, given a $3$-manifold $M$ with an embedded link $L\subset M$, defect skein theory defines the $R$-module $\sk_{\AM}(M,L)$ as the $R$-linear span of ribbon graphs in $M$ that are $\A$-labelled in $M\setminus L$ and $\M$-labelled on $L$, subject to $\A$-local relations away from the defect $L$ and $\M$-local relations near the defect (see Section~\ref{sec:skeins-def} for details). 
Moreover, to a surface $\Sigma$ with a set of defect points $P\subset \Sigma$ the functor in Theorem~\ref{thm:intro-def-sk-TQFT} assigns the (defect) skein category $\sk_{\AM}(\Sigma,P)$ which we further show to be equivalent to the associated stratified factorization homology (see Proposition~\ref{prop:sk-cat-fact}), \ie 
\begin{equation*}
    \sk_\AM(\Sigma,P) \simeq \int_{(\Sigma,P)}{\AM}~.
\end{equation*}

\subsection{Defect skein theory via quantum symmetric pairs}
\label{subsec:intro-QSP-defects}

A rich source of braided module categories is supplied by the quantizations of symmetric subgroups and their associated symmetric spaces. Since the foundational observation of Koornwinder \cite{Koornwinder} that the quantum analogues of symmetric subgroups $G^\theta$ of a reductive group $G$ should not correspond to Hopf subalgebras of $\mathcal{U}_q\g$ but rather to \textit{coideal subalgebras}, families of examples were discovered by the pioneering works \cite{Noumi-Sugitani, Noumi, DNS, Let02, Let03} from the perspective of deformations of rings of functions on the associated symmetric space. 

From a representation-theoretic point of view, it is indeed the structure of a coideal subalgebra on $\mathcal{U}_{\mathbf{c},\mathbf{s}}\mathfrak{g}^\theta \subset \mathcal{U}_q \g$ and not that of a Hopf subalgebra which endows the category $\mathrm{Rep}_q(G^\theta)$ with the structure of a module category over $\mathrm{Rep}_q(G)$\footnote{In the case when the symmetric pair is outer, a $\Z_2$-equivariantization is necessary.}. The coideal subalgebra $\mathcal{U}_{\mathbf{c},\mathbf{s}}\mathfrak{g}^\theta$ depends on a family of parameters $\mathbf{c},\mathbf{s}$ depending on the type and subject to conditions (see Section~\ref{subsec:QSP}). In a series of works \cite{BK, Kol14, Kol20} of Kolb and Balagovic--Kolb, this $\mathrm{Rep}_q(G)$-module was successfully quantized as a \textit{braided module} (or equivalently as an $E_2$-module over the $E_2$-algebra $\mathrm{Rep}_q(G)$). Algebraically, the braiding is encoded in a \textit{universal $K$-matrix} $\mathcal{K} \in \mathcal{U}_{\mathbf{c},\mathbf{s}}\g^\theta \otimes \mathcal{U}_q\g$ which gives an automorphism of the module action: for every $a \in \mathrm{Rep}_q(G)$ and every $m \in \mathrm{Rep}_q(G^\theta)$, we have a braiding
$$\mathcal{K}_{m \ract a}: m \ract a \longrightarrow m \ract a$$
appropriately compatible with the $R$-matrix of $\mathrm{Rep}_q(G)$. Topologically, the element $\mathcal{K}$ encodes local skein relations near a line defect: a skein labelled by a representation $V \in \mathrm{Rep}_q(G)$ should pick up a coupon labeled by $\mathcal{K}_V$ when braided around a defect line. 
These structures have also been studied under the name of \textit{$\iota$quantum groups} \cite{BaoWang1,BaoWang2, BaoShanWangWebster}. 

Given Witten's finiteness conjecture on $G$-skein modules of closed 3-manifolds \cite{GJS} (established for $G = \SL_2$ \cite{Belletti-Detcherry, JR}), one of our motivations to study skein theory with quantum symmetric pair (QSP) defects was the following natural question:
\begin{question}\label{question:finiteness}
Let $M$ be a closed 3-manifold with a defect knot $L\subset M$. What type of defect labels $\M$ (braided module categories over $\A = \Rep_qG$) guarantee that the defect skein module $\sk_{(\Rep_q(G),\M )}(M,L)$ is finite dimensional over $\C(q)$?
\end{question}
The transparent defect $\M = \A$ satisfies the above property since the defect skein module reduces to the non-defect case, but there are plenty of counterexamples as well: take for instance the braided $\A$-module $\int_{\Ann}{\A}$, which gives rise to the (infinite dimensional) $G$-skein module of the open 3-manifold $M \setminus L$. 

Towards answering this question, we establish a crucial gluing property of defect skein theory which relies on the TQFT picture of Theorem \ref{thm:intro-def-sk-TQFT}: for $q$ generic, the skein module may be expressed as 
\[\sk_{(\Rep_q(G),\M )}(M,L)\cong \sk^{\mathrm{int}}_{\Rep_q(G)}(M\setminus L)\otimes_{\DqG} \mathcal{L}_\M,\]
the relative tensor product of the internal skein module of the knot complement $M\setminus L$ and the internal skein module $\mathcal{L}_\M$ of the solid torus with defect core labelled by $\M$, over the internal skein algebra of the punctured torus $\DqG$. Following \cite{JR}, finiteness of the relative tensor product is closely related to the holonomicity of the respective modules. In this direction, we prove 
\begin{mainthm} [Theorem \ref{thm:QSP-holonomic}]\label{thm:holonomicity of LM}
    The $\DqG$\footnote{A $\Z_2$-equivariantization is needed in the outer case.}-module $\mathcal{L}_\M$ is holonomic for braided module categories $\M$ associated to QSPs.  
\end{mainthm}
The finiteness of the Kauffman bracket defect skein module (Theorem \ref{thm:def-skmod-finiteness}) is then a direct consequence of the preceding Theorem \ref{thm:holonomicity of LM} and the holonomicity results of \textit{op. cit}.

\subsection{Equivariant skein theory and double affine Hecke algebras}

It is well-known, after \cite{GJVtori}, that skein theory on the 2-torus gives rise to modules over the double affine Hecke algebra (DAHA) of Type $\mathrm{A}$ with suitably specialized parameters. 

In \cite{Jordan-Ma}, analogous constructions for the DAHA of Type $\mathrm{C}^\vee \mathrm{C}$ where obtained via purely algebraic means. Towards a topological underpinning of these DAHA modules, Weelinck has developed the theory of \textit{$\Z_2$-equivariant factorization homology} \cite{Wee19,Wee20} in anticipation that skein theory of the $\Z_2$-orbifold torus (with hyperelliptic involution) would play the same role that the 2-torus played in Type $\mathrm{A}$. As a secondary application of our defect skein theory, we argue that $\Z_2$-equivariant skein theory with coefficients in a QSP may be regarded as defect skein theory on the $\Z_2$-quotiented bulk, with defect being the branch locus. These considerations lead to a topological construction of the DAHA modules of Jordan--Ma (see Proposition \ref{prop: DAHA}), and suggest generalizations for DAHAs of other types. We leave a thorough algebraic investigation of the resulting DAHA modules to future work.

\subsection{Relation to recent developments and outlook} 
\label{subsec:intro-discussion}

\subsubsection{Relative Langlands duality}
\label{subsec:intro-rel-Langlands}
Boundary conditions for 4-dimensional $\mathcal{N} = 4$ supersymmetric Yang--Mills theory and consequences of their $S$-duality \cite{Kapustin-Witten, Gaiotto-Witten} play a central role in the emergent arithmetic theory of \textit{relative Langlands duality} in the sense of Ben-Zvi--Sakellaridis--Venkatesh \cite{BZSV}. Metaphorically (or slightly more precisely, via the \textit{arithmetic TQFT} philosophy), one regards a global field $F$ (\ie a number field or the function field of an algebraic curve over a finite field) as an arithmetic 3-manifold, and considers the evaluation of two $S$-dual theories on $F$ to obtain an isomorphism of Hilbert spaces of states. For the purposes of this motivational discussion, we shall refer to these two theories as $\mathcal{A}_G$ and $\mathcal{B}_{\check{G}}$, for a pair of Langlands dual groups $G$ and $\check{G}$. 

The spaces of states assigned by $\mathcal{A}_G$ and $\mathcal{B}_{\check{G}}$ to $F$ admit the following arithmetic descriptions: on the one hand, one considers the space of automorphic functions for a reductive group $G$ defined over $F$,
$$\mathcal{A}_G(F) = \text{ (unramified) automorphic functions on } G/F,$$
and on the other hand, one considers the space of algebraic distributions on the moduli stack of $\check{G}$-valued unramified Galois representations
$$\mathcal{B}_{\check{G}}(F) = \text{ algebraic distributions on } \big\{\pi_1^{\mathrm{\acute{e}t}, \mathrm{ur}}(F) \to \check{G}\big\}.$$
The expected canonical identification of the two vector spaces $\mathcal{A}_G(F)$ and $\mathcal{B}_{\check{G}}(F)$ is the content of the \textit{global arithmetic Langlands correspondence}, which behaves essentially as a spectral decomposition, with $\check{G}$-valued unramified Galois representations labeling the eigenvalues of automorphic functions under the action of a globally defined commutative Hecke algebra. 

One fundamental insight of Ben-Zvi--Sakellaridis--Venkatesh is the identification of the role of \textit{Hamiltonian actions} in this arithmetic TQFT. Given a nice Hamiltonian $G$-space $X$ and a nice Hamiltonian $\check{G}$-space $\check{X}$, one can construct distributions
\begin{equation} \label{eq:arithm-bnd-cond}
    \mathcal{A}_G(X):\mathcal{A}_G(F) \longrightarrow \C \, \text{ and } \,\mathcal{B}_{\check{G}}(\check{X}): \mathcal{B}_{\check{G}}(F) \longrightarrow \C
\end{equation}
which, in many cases, $\mathcal{A}_G(X)$ may be interpreted as an \textit{automorphic period} and $\mathcal{B}_{\check{G}}(\check{X})$ may be interpreted as a \textit{Langlands $L$-function}.  Loosely speaking, the pairs $(G,X)$ and $(\check{G}, \check{X})$ are relative Langlands dual if $\mathcal{A}_G(X)$ and $\mathcal{B}_{\check{G}}(F)$ coincide under the Langlands correspondence $\mathcal{A}_G(F) \simeq \mathcal{B}_{\check{G}}(F)$; in particular, given a Hecke eigenform $f$ with Hecke eigenvalues labeled by a $\check{G}$-valued Galois representation $\varphi_f$, we expect an \textit{automorphic period formula}
\begin{equation} \label{eq:period-formula}
    \langle \mathcal{A}_G(X), f\rangle = \langle \mathcal{B}_{\check{G}}(\check{X}), \varphi_f\rangle
\end{equation}
as a formal consequence of the Plancherel theorem. 

While the precise relationship between skein theory and the arithmetic form of Langlands' theory remains to be understood, one suspects that a suitable deformation of the theories $\mathcal{A}_G$ and $\mathcal{B}_{\check{G}}$ should be analogous to skein theory; in fact, the two theories should become completely topological with respect to generic deformation parameters, and the descriptions of the two sides should become more symmetric. In particular, the deformed theories $\mathcal{A}_{G,q}$ and $\mathcal{B}_{\check{G}, \check{q}}$ where $q$ and $\check{q}$ are Langlands dual deformation parameters participate in a conjectural \textit{skein Langlands duality} \cite{JLanglands}: for certain closed oriented 3-manifolds $M$, one expects an identification
\begin{equation}\label{eq:sk-Langlands}
    \mathcal{A}_{G,q}(M) = \sk_G(M) \overset{?}{\longleftrightarrow} \sk_{\check{G}}(M) = \mathcal{B}_{\check{G}, \check{q}}(M)
\end{equation}
analogous to the spectral decomposition of Langlands. While the validity of \eqref{eq:sk-Langlands} is far from being understood - indeed, at the time of writing, the available supporting evidence towards \eqref{eq:sk-Langlands} have been dimension counts without an isomorphism being constructed - it still makes sense to ask the following motivating question:
\begin{question} \label{question:sk-arithm-bnd-cond}
    What is the analogue of equation \eqref{eq:arithm-bnd-cond} in skein theory?
\end{question}
The present work represents a tiny step towards answering, or formulating, the preceding question. In particular, one must understand how the Hamiltonian actions labelling boundary conditions in relative Langlands duality may be deformed to an object intrinsic to skein theory. 

Casting aside the question of Langlands duality for now, in some sense the response to Question \ref{question:sk-arithm-bnd-cond} is immediate: skein theory, being a fully topological QFT, has well-defined notions of boundary theories (or at least, it has well-established expectations for boundary theories). We will work one dimension lower topologically than in the arithmetic setting.\footnote{Indeed, to metaphorically make sense of \eqref{eq:arithm-bnd-cond} one considers 4-dimensional cylinders over the 3-manifold ``$\spec F$", and skein theory is not currently equipped to evaluate on such cobordisms.} As a first approximation, one searches for the following structures: given a Hamiltonian $G$-space $X$, we should assign a categorical distribution
\begin{equation} \label{eq:cat-distribution}
    \sk_G(X) : \sk_G(\Sigma) \longrightarrow \mathrm{Vect}
\end{equation}
for every $M$ an oriented 3-manifold with boundary surface $\Sigma$, as the analogue of \eqref{eq:arithm-bnd-cond}. 

The main constructions of the present project (see \S 3.1 and in particular Definition \ref{def:TQFT-line-def} and Definition \ref{def:skmod}) can be understood as such a categorical distribution when $\Sigma$ is a genus 1 surface, in which case what we termed \textit{skein module with defect} coincides with the value of $\sk_G(X)$ on the empty object of $\sk_G(\Sigma)$
$$\sk_{G,X}(M,\Sigma) := \sk_G(X)(\varnothing).$$
When $\Sigma$ is the boundary torus of a normal neighborhood $\nu(K)$ of an embedded knot $K$, we regard $\sk_{G,X}(M,\nu(K))$ as a vector space-valued knot invariant for $K \subset M$. 

\subsubsection{Deformation, quantization, and shifted symplectic geometry}
\label{subsec:intro-defquant-shiftedsympl}

Unfortunately but interestingly, the proposal sketched by \eqref{eq:cat-distribution} -- that is, the assignment of a categorical distribution to a Hamiltonian $G$-space -- cannot be literally correct without further engineering. Indeed, underlying the skein theory TQFT is the $E_2$-algebra, or braided monoidal category $\mathrm{Rep}_q(G)$. As such, we understand that a boundary theory with which one can label embedded knots inside closed 3-manifolds should arise from $E_2$-modules, or braided modules over $\mathrm{Rep}_q(G)$, which \textit{a priori} appear unrelated to the Hamiltonian actions considered by the relative Langlands program.

After some reflection, the link between these two notions of boundary conditions can be constructed by combining some lessons from the work of Ben-Zvi--Brochier--Jordan and Safronov, the most relevant to us being the following:
\begin{itemize}
    \item \cite[Theorem 1.1]{BZBJ18b}. The structure of a braided module over $\mathrm{Rep}_q(G)$ can be repackaged as a certain $q$-deformation of a $G$-valued \textit{quasi-Hamiltonian moment map} in the sense of \cite{AMM97}.
    \item \cite[\S 2.2 and \S 2.3]{Safronov1}. A $\mathfrak{g}^*$-valued Hamiltonian moment map (resp. a $G$-valued quasi-Hamiltonian moment map) is equivalent to a 1-shifted Lagrangian in $\mathfrak{g}^*/G$ (resp. in $G/G$). 
\end{itemize}
The path from the Hamiltonian actions considered by Ben-Zvi--Sakellaridis--Venkatesh to boundary conditions for skein theory can thus be schematically represented as follows:
\begin{equation}\label{eq:Hamiltonian-qsHamiltonian-brmod}
    \text{Hamiltonian actions} \overset{\text{deformation}}{\longsquiggly} \text{quasi-Hamiltonian actions} \overset{\text{quantization}}{\longsquiggly} \text{Braided modules}
\end{equation}
At our current level of understanding, both the existence and the uniqueness of the ``deformation" and ``quantization" arrows are unclear. It is nonetheless evident that, given a Hamiltonian $G$-action $X$ with moment map $\mu$, the squiggly arrows represent the following precise pieces of extra data:
\begin{itemize}
    \item (Quasi-Hamiltonian deformation). Let $\mathbf{G}$ be the deformation to the normal cone along $\mathrm{B}G \to G/G$; then $\mathbf{G}$ is a 1-shifted symplectic stack over $\mathbf{A}^1$ with special fiber $\g^*/G$ and general fiber $G/G$. A \textit{quasi-Hamiltonian deformation} of the Hamiltonian moment map $\mu: X/G \to \g/G$ means the data of a $G$-space $\mathbf{X}$ over $\mathbf{A}^1$ equipped with a 1-shifted Lagrangian morphism $\boldsymbol{\mu}:\mathbf{X}/G \to \mathbf{G}$ over $\mathbf{A}^1$ such that the base change to $0 \in \mathbf{A}^1$ is identified with $\mu$. 
    \item (Shifted geometric quantization). Suppose we are given a quasi-Hamiltonian moment map $\mu_1: X_1/G \to G/G$, which we may think of as arising from taking the fiber at $1 \in \mathbf{A}^1$ from the preceding construction. Then a \textit{shifted geometric quantization} is the data of a quantum moment map deforming $\mu_1$ in an appropriate sense, and providing the structure of a braided module over $\mathrm{Rep}_q(G)$.
\end{itemize}
While it is our intention to study more systematically in the future these two deformation problems, the first of geometric and the second of categorical nature, in the present project we content ourselves with a class of examples where the quasi-Hamiltonian deformation and shifted geometric quantization should be well-known to experts: we consider Hamiltonian $G$-actions of the form
$$G \acts X = T^*(G^\theta \backslash G) \text{ where } \theta \text{ is an involution on } G.$$
In other words, $G^\theta \subset G$ is a symmetric subgroup. In this case, 
\begin{itemize}
    \item A quasi-Hamiltonian deformation\footnote{This example will appear explicitly in forthcoming work of the first author with Š. Kaubrys \cite{EricSarunas}.} of $X$ is given (above $1 \in \mathbf{A}^1$) by the multiplicative analogue of $X$, 
    $$G \acts X_1 := G^\theta \backslash G \times^{G^\theta} G.$$
    \item The shifted geometric quantization of $X_1$ is constructed via the theory of QSPs \cite{BK, Kol14, Kol20}, directly as a braided module over $\mathrm{Rep}_q(G)$.
\end{itemize}
Before approaching the question of \textit{relative Langlands duality} in this context, one must first construct and understand many examples in which the deformation problems posed above can be solved. 

\subsection{AI usage disclosure}

This research was conducted, and this paper was drafted, without the use of AI tools, beyond basic literature search.  We conducted one AI generated referee report for typos and minor errors which we corrected by hand.

\subsection{Acknowledgements}
E.Y.C. would like to recognize the support of the Swiss National Science Foundation No. 196960 and the JSPS Postdoctoral Fellowship during the completion of this project. D.J.\ and I.R.\ were supported the Simons Collaboration on Global Categorical Symmetries 1013836, and by EPSRC Open Fellowship “Complex Quantum Topology”, grant number EP/Y008812/1. I.R\ was also supported by the ERC under the EU’s Horizon 2020 programme grant agreement No 948885.

\subsubsection*{Conventions:}
\begin{itemize}
    \item Let $R$ will denote a fixed commutative ring. We will consider various (balanced) braided tensor categories $\A$ linear over $R$.
    \item Let $G$ be a reductive group. In the case when $\A = \mathrm{Rep}_q(G)$, \ie, Type I representations of the quantum group of $G$, we take $R = \C[q^{\pm 1/d}]$. Here $d \in \N$ denotes the smallest positive integer such that $(\mu,\nu)\in \frac{1}{d}\Z$ for all weights $\mu,\nu$ (e.g.\ $d=n+1$ in type $\mathrm{A}_n$). 
    \item We say that $q$ is \textit{generic} to indicate we are tensor over $k = \mathrm{Frac}(R) = \C(q^{1/d})$. Whenever we discuss \textit{specializable parameters} we will work over the localized ring $R_{(q-1)}$, so that certain deformation elements can be specialized as $q \to 1$.
\end{itemize}

\section{Prerequisites}
\label{sec:pre}

In this section, we recall some categorical background and the notion of braided module categories. We then  recall how quantum symmetric pairs give rise to examples of braided module categories. 

\subsection{Categorical background}\label{subsec:cat-background}

In this section, we briefly recall the relevant categorical notions involved and fix conventions. For a more extensive exposition see \cite{Cooke}, \cite{BrownHaioun} and \cite{GJS}.  

Let $\Cat$ denote the 2-category of (small) $R$-linear categories, with $R$-linear functors and natural transformations as morphism categories
$$\Hom_{\Cat}(\A,\mathcal{B}) := \Fun_{R}(\A, \mathcal{B}).$$ The $R$-linear tensor product $\boxtimes_R$ of $R$-linear categories endows $\Cat_R$ with a symmetric monoidal structure, with the category of $R$-modules $\Mod_R$ as the monoidal unit. 

\begin{definition}\label{def:bimod}
    The bicategory $\Bimod$ of $R$-linear categories and bimodules consists of small $R$-linear categories as objects and morphisms categories $$\Hom_{\Bimod}(\A,\mathcal{B}):= \Fun_{R}(\mathcal{B}\boxtimes\mathcal{A}^\opp, \Mod_R)~.$$
\end{definition}
A bimodule $F: \mathcal{B}\boxtimes \mathcal{A}^\opp \to \Mod_R$ is alternatively called a \textit{profunctor} and denoted by $F: \mathcal{A}\proarrow \mathcal{B}$. Composition of $F: \mathcal{B} \boxtimes \mathcal{A}^\opp \rightarrow \Mod_R$ with $G: \mathcal{C}\boxtimes \mathcal{B}^\opp \rightarrow \Mod_R$ is given by the coend \cite[Ch.\ 7.8]{Bor94}: 
$$G\circ F:= \int^{b\in \mathcal{B}}{G(-,b)\otimes F(b,-)}: \mathcal{C}\boxtimes \mathcal{A}^\opp \rightarrow \Mod_R~.$$
The $R$-linear tensor product $\boxtimes$ equips $\Bimod$ with a symmetric monoidal structure \cite[Sec.\ 7]{DS97}. There is a symmetric monoidal faithful functor $\Cat \to \Bimod$ which is the identity on objects and maps a functor $F: \mathcal{A}\to \mathcal{B}$ to the bimodule $\Hom_{\mathcal{B}}(-,F(-)):\mathcal{B}\boxtimes \mathcal{A}^\opp \to \Mod_R$. 

Finally, we define the bicategory of locally presentable $R$-linear categories; we refer the reader to \cite{BCJF} for details. 
\begin{definition}\label{def:prl}
    The bicategory $\prl$ consists of locally presentable $R$-linear categories as objects and cocontinuous functors and natural transformations as morphism categories $$\Hom_{\prl}(\mathcal{A}, \mathcal{B}):= \Fun_{cc}(\mathcal{A},\mathcal{B})~.$$
\end{definition}
The ordinary $R$-linear tensor product of two categories in $\prl$ is no longer locally presentable and thus we instead use the Deligne--Kelly product $\boxtimes$ (retaining the same notation as it will be clear from context) which provides $\prl$ with a symmetric monoidal structure. The \textit{free cocompletion} $\Ahat:= \Fun_R(\mathcal{A}^\opp,\Mod_R)= \Hom_{\Bimod}(\mathcal{A}, \Mod_R)$ extends to a symmetric monoidal fully faithful functor 
\begin{equation}\label{eq:free-cococompletion}
\ind{(-)}: \Bimod \to \prl
\end{equation}
which identifies $\Bimod$ with the full subcategory in $\prl$ spanned by categories with enough compact projectives. Under this identification, if $\mathcal{C}\in \prl$ has enough compact projectives, then $\mathcal{C}\simeq \widehat{(\mathcal{C}^\mathrm{cp})}$ where $\mathcal{C}^\mathrm{cp}\subset \mathcal{C}$ is the subcategory of compact projective objects, or cp-objects for short.

Since we have a symmetric monoidal functor $\Cat\to \Bimod\to \prl$, a tensor structure (equivalently $E_1$-algebra structure) on a $R$-linear category $\mathcal{A}\in \Cat$ induces a tensor structure on $\Ahat$ given by \textit{Day convolution} \cite{Day}: for $X,Y\in \ind{\mathcal{A}}$ we have 
\begin{equation}\label{eq:Day-convolution}
X\otimes Y:= \int^{z,w\in \mathcal{A}}{\Hom_{\mathcal{A}}(-,z\otimes w)\otimes X(z)\otimes Y(w)}~.    
\end{equation}
The tensor unit is given by $\unit\in \A\subset \Ahat$ where the latter inclusion is the \textit{Yoneda embedding}. If $\A$ is additionally braided (equivalently a framed $E_2$-algebra) resp. ribbon (equivalently an oriented $E_2$-algebra), then so is $\Ahat$. 
\begin{rem}\label{rem:E2-mod-coco}
The free cocompletion is also compatible with $E_1$- and, more importantly for us, $E_2$-module structures over $E_2$-algebras, which we will see later in the discussion of braided module categories. 
\end{rem}
An algebra in $\Ahat$ is the same as a lax monoidal functor $X: \A^\opp \to \Mod_R$, \ie there are natural transformations $$\varphi_{z,w}:X(z)\otimes X(w) \to X(z\otimes w)$$
and $$\eta:R \to X(\unit)$$ compatible with associators and unitators of $\A$. A (left) module $Y\in \Ahat$ over such an algebra $X\in \Ahat$ is equipped with a natural transformation $$X(z) \otimes Y(w) \to Y(z\otimes w)$$ compatible with the lax monoidal structure of $X$. 

\subsection{Braided module categories}
\label{subsec:brmod} 
Let $\A$ be a tensor category, \ie an $E_1$-algebra in $\Cat$. 

\begin{definition}\label{def:mod-cat}
An $\A$-module category is a small $R$-linear category $\M$ together with an action functor $\lact: \A \times \M\to \M$, which is $R$-bilinear, as well as natural transformations $a_{x,y,m}:(x\otimes y)\lact m \cong x\lact(y\lact m)$ and $u_{m}:\unit \lact m \cong m$ compatible with the associators and unitators of $\A$. 
\end{definition}
The preceding definition coincides with the notion of an $E_1$-module in $\Cat$ of the $E_1$-algebra $\A$. 

Let $m\in \M$ be an object and let
\[\mathrm{act}_m:= (-)\lact m: \A\to \M\] 
denote the associated action functor. Passing to the free cocompletions of $\A$ and $\M$ we obtain a right adjoint\footnote{All 1-morphisms in $\prl$ admit right adjoint functors, albeit not necessarily cocontinuous.} to $\mathrm{act}_m$ \[\mathrm{act}_m^R:= \ihom(m,-): \ind{\M}\to \ind{\A}~.\] 
If $\mathrm{act}_m^R$ is conservative (resp.\ cocontinuous), then we call $m$ an $\A$-\textit{generator}\footnote{Note that if $\mathrm{act}_{m}^R$ is conservative then $\mathrm{act}_m$ is \textit{dominant}, \ie every object in $\M$ is a quotient (equivalently subobject if $\M$ is semisimple) of $x\lact m$ for some object $x\in \A$.} (resp.\ $\A$-\textit{projective}). If both, then we call $m$ an $\A$-\textit{progenerator}. We recall the reconstruction theorem for module categories: 

\begin{prop}[Theorem\ 4.6 of \cite{BZBJ18a}]
\label{prop:module-cat-reconstruction}
    Let $\unit_\M\in \M$ be an $\A$-progenerator. Then, we have an equivalence of $\A$-module categories
    \[\ind{\M}\simeq \Mod\mbox{-}\iend(1_\M)~\]
    where $\iend(\unit_\M)= \mathrm{act}_{\unit_\M}^R\mathrm{act}_{\unit_\M}(\unit_\A)\in \ind{\A}$ is the internal endomorphism ring associated to the adjunction.
\end{prop}

Let $\A$ denote a braided tensor category with braiding $c$, aka a (framed) $E_2$-algebra in $\Cat$ (resp.\ $\prl$). For the purpose of this paper, it suffices to work in $\Cat$ as the categories we are interested in have enough compact projectives and are even cp-rigid, \ie all cp-objects are rigid/dualizable. 

\begin{definition}\label{def:br-mod-cat}
A \textit{braided module category} over $\A$ is a left $\A$-module category $\M$ equipped with a natural automorphism $e$ of the action functor $\lact: \A\times \M\rightarrow \M$ such that for all $x,y \in \A$ and $m\in \M$ we have the following compatibility conditions:
    \begin{align}
        e_{x, y\lact m} &= (c_{y,x}\lact \id_{m})\circ (\id_{y}\lact e_{x,m})\circ (c_{x,y}\lact\id_{m})\\   
        e_{x\otimes y, m} &= (c_{x,y}^{-1}\lact \id_{m})\circ(\id_{y}\lact e_{x,m})\circ (c_{x,y}\lact \id_{m})\circ (\id_{x}\lact e_{y,m})\\ 
        & = (\id_{x}\lact e_{y,m})\circ (c_{x,y}^{-1}\lact \id_{m})\circ(\id_{y}\lact e_{x,m})\circ (c_{x,y}\lact \id_{m})~.\notag
    \end{align}
    where the associators of $\A$ and $\M$ are suppressed for legibility.
\end{definition}
The braided module structure $e_{x,m}$ can be understood graphically as a braiding of an $x$-labelled strand with a special (colored red) $m$-labeled strand as follows: 
\begin{equation}\label{eq:braided-structure-graphically}
    e_{x,m} = \begin{tikzpicture}[baseline={([yshift=-.5ex]current bounding box.center)}]
        \brmod{0,0}{.5}{1.2}
        \node[anchor=north] at (0,0) {\scriptsize$x$};
        \node[anchor=north] at (.5,0) {\scriptsize$m$};
        \node[anchor=south] at (0,1.2) {\scriptsize$x$};
        \node[anchor=south] at (.5,1.2) {\scriptsize$m$};
    \end{tikzpicture}
\end{equation}

\begin{rem}\label{rem:brmod-convention}
Our definition of braided module category coincides with the convention of \cite{BZBJ18b} by adapting from right module to left module structures. In particular, axiom (2) differs slightly from the convention used in \cite{Bro13} and \cite{Kol20}.
\end{rem}
\begin{rem}
\label{rem:brmod-Morita}
A braided module category is the same as an $E_2$-module over an $E_2$-algebra in $\Cat$ or equivalently a $2$-endomorphism in $\Aut_{\operatorname{BrTens}}({}_{\A}\A_\A)$ of the identity 1-morphism ${}_{\A}\A_\A: \A \rightarrow \A$ in $\operatorname{BrTens}$. Here, $\operatorname{BrTens}$ denotes the Morita category of braided tensor categories, $E_1$-bimodules, $E_2$-bimodules etc.\ (see \cite{BJS}). 
    
The identification of braided $\A$-module categories and $E_2$-module categories is found in \cite[Theorem\ 3.11]{BZBJ18b}. Phrased in terms of \textit{factorization homology} (see Section~\ref{subsec:fact-hom}) this is the same as an $\int_{\Ann}\A$-module structure, \ie an $E_1$-module over the $E_1$-algebra $\int_{\Ann}{\A}$.   
\end{rem}

If, in addition, $\A$ is balanced\footnote{We reserve the term \textit{ribbon} for balanced braided tensor categories which are cp-rigid and the balancing is compatible with rigidity: $\theta_{x^\ast} = \theta_x^\ast$ for any cp-object $x$.} with twist $\theta$, then a braided module category $\M$ is called \textit{balanced} if it is also equipped with a natural automorphism $\phi: \Id_{\M}\simeq \Id_\M$, called \textit{balancing}, such that $\phi_{x \lact m} = e_{x,m} \circ (\theta_x \lact \phi_m).$ In view of \cite[Theorem\ 3.12]{BZBJ18b} there is a canonical balancing of $\M$ given a balancing of $\A$. 

\begin{rem}\label{rem:brmod-E2}
    A \textit{balanced} braided tensor category $\A$ is the same as an \textit{oriented} $E_2$-algebra in $\Cat$ and a \textit{balanced} braided module category is the same as an \textit{oriented} $E_2$-module over the latter. Hence, rephrasing the above, an $E_2^\mathrm{fr}$-module over the underlying $E_2^\mathrm{fr}$-algebra of an $E_2^\mathrm{or}$-algebra extends canonically to an $E_2^\mathrm{or}$-module structure.
\end{rem}

In the presence of the additional structure of (balanced) braided module categories, the reconstruction result of Proposition~\ref{prop:module-cat-reconstruction} enhances to the following: 
\begin{prop}[Theorem\ 4.3 of \cite{BZBJ18b}]\label{prop:braided-module-cat-reconstruction}
    Let $\M$ be a (balanced) braided $\A$-module category over the (balanced) braided tensor category $\A$ and $\unit_\M\in \M$ be an $\A$-progenerator. Then, we have an equivalence of (balanced) braided $\A$-module categories
    \[\widehat{\M}\simeq \Mod_{\A}\mbox{-}\iend(\unit_\M)~.\]
\end{prop}
The braided structure on the latter is induced by a canonical algebra homomorphism 
\begin{equation*}
    \mu:\Ocal_\A\to\iend(\unit_\M)
\end{equation*}
out of the \textit{reflection equation algebra} $\Ocal_\A\in \A$ called the \textit{quantum moment map}. We will recall more details in Sections \ref{subsec:ISA} and \ref{subsec:qmm-QSP}.

\subsection{Quantum symmetric pairs}
\label{subsec:QSP}

Our main source of examples of braided module categories arises from symmetric pairs $(G,G^\theta)$, which consist of a reductive group $G$ and a group involution $\theta \in \mathrm{Aut}(G)$ for which $G^\theta$ is the fixed subgroup. For now, we shall assume that the involution $\theta$ is \textit{inner}, that is $\theta$ is obtained by the adjoint action of some element $J \in G$ (which necessarily satisfies the condition that $J^2$ is central in $G$).

The quantizations of symmetric pairs are referred to as \textit{quantum symmetric pairs} in the sense of Letzter \cite{Let99}, and the resulting structure may be regarded as a deformation of the $\mathrm{Rep}(G)$-module category 
\begin{equation} \label{eq:deformation-of-module-cat}
    \mathrm{Rep}(G) \acts \mathrm{Rep}(G^\theta) \rightsquigarrow \mathrm{Rep}_q(G) \acts \mathrm{Rep}_{\mathbf{c},\mathbf{s}}(G^\theta)
\end{equation}
where $\mathbf{c} = (c_i),\mathbf{s} =(s_j)$ are certain tuples of scalars in $R$ that represent additional deformation parameters. In \cite{Kol20}, the module category $\Rep_{\mathbf{c}, \mathbf{s}}(G^\theta)$ is further endowed with the structure of a braided module over $\Rep_q(G)$ (or a $\Z_2$-equivariantization thereof if the involution is outer). 

To make precise sense of \eqref{eq:deformation-of-module-cat}, we will only consider parameters $\mathbf{c}, \mathbf{s}$ that are \textit{specializable}, \ie that they lie in the ring $R_{(q-1)}$ and $c_i(1) = 1$, in which case we have the following theorem of Kolb. 

\begin{thm}[Theorem 10.8 of \cite{Kol14}] \label{thm:Kolb-specializability}
    Let $\mathbf{c}, \mathbf{s}$ be specializable parameters. Then $\mathrm{Rep}_{\mathbf{c}, \mathbf{s};q=1}(G^\theta) \simeq \mathrm{Rep}(G^\theta)$ as module categories over $\mathrm{Rep}_{q=1}(G) \simeq \mathrm{Rep}(G)$.
\end{thm}

Formally speaking, quantum symmetric pairs arise from \textit{quasi-triangular Hopf algebras} and \textit{quasi-triangular coideal subalgebras}, whose module categories provide the braided module categories we aim to recall (see \cite{Kol20} for a detailed exposition). To summarize these notions for the reader's convenience, let $H$ be a quasi-triangular Hopf algebra and $\A= H\modu$ its category of modules which admits the structure of a braided monoidal category. 

\begin{definition}\label{def:quasi-triangular-coideal}
    Let $H$ be a quasi-triangular Hopf algebra with universal $R$-matrix $\mathcal{R} \in H$. A subalgebra $B\subset H$ is called a \textit{(left) coideal} if $\Delta(B)\subset B\otimes H$. A coideal subalgebra $B$ is called \textit{quasi-triangular} if there is an invertible element $\mathcal{K}\in B\otimes H$ such that 
    \begin{align}
    & \mathcal{K} \Delta(b) = \Delta(b)\mathcal{K}\quad\forall b \in B~, \\ 
    & (\Delta\otimes \id)(\mathcal{K}) = \mathcal{R}_{32}\mathcal{K}_{13}\mathcal{R}_{23}~,\\ 
    &(\id\otimes \Delta)(\mathcal{K}) = \mathcal{R}_{32} \mathcal{K}_{13} \mathcal{R}_{23} \mathcal{K}_{12}~.
    \end{align}
    The element $\mathcal{K}$ is then called the \textit{universal $K$-matrix}.
\end{definition}
The category $\M= B\modu$ of $B$-modules inherits naturally the structure of an $\A$-module category when $B$ is a coideal. If in addition, $B$ is quasi-triangular with a universal $K$-matrix $\mathcal{K}\in B\otimes H$, then $\M$ becomes braided with $e$ given by multiplication by $\mathcal{K}$. 

\begin{rem}\label{rem:K-matrix-vs-J-matrix}
    The datum of a universal $K$-matrix $\mathcal{K}\subset B \otimes H$ for a coideal subalgebra $B\subset H$ is equivalent, by \cite[Lemma 2.9]{Kol20}, to the datum of an invertible element $ \mathcal{J} \in H$ (or a suitable completion thereof) satisfying: 
\begin{align}
    & \mathcal{J} b  = b \mathcal{J} \quad\forall b \in B~, \\ 
    & \mathcal{R}_{21}\mathcal{J}_{2}\mathcal{R}_{12}\in B\otimes H,\\ 
    &\Delta(\mathcal{J}) = \mathcal{R}_{21} \mathcal{J}_2 \mathcal{R}_{12} \mathcal{J}_1~.
\end{align}
The equivalence is given by mapping a universal $K$-matrix $\mathcal{K}$ to $\mathcal{J}:= (\varepsilon\otimes \id)(\mathcal{K})\in H$ and, conversely, $\mathcal{J} \in H$ to $\mathcal{K}:= \mathcal{R}_{21}\mathcal{J}_{2}\mathcal{R}_{12}$. We will abuse terminology and refer to both $\mathcal{J}$ and $\mathcal{K}$ as the universal $K$-matrix.
\end{rem}

Theorem~\ref{thm:Kolb-specializability} is provided by deforming the coideal subalgebra\footnote{This algebra is denoted by $B_{\mathbf{c},\mathbf{s}}$ in \cite{Kol20}.} by specializable parameters $\mathbf{c},\mathbf{s}$  
$$\mathcal{U}\mathfrak{g}^\theta \subset \mathcal{U}\mathfrak{g} \, \rightsquigarrow \, \mathcal{U}_{\mathbf{c}, \mathbf{s}} \mathfrak{g}^\theta \subset \mathcal{U}_q\mathfrak{g}~,$$
with the algebras on the right hand side being appropriate integral forms over the ring $R$. That is,
\[\C\otimes_{R}\left(\mathcal{U}_{\mathbf{c}, \mathbf{s}} \mathfrak{g}^\theta\right)\cong \mathcal{U}\mathfrak{g}^{\theta}~.\]
The coideal subalgebra $\mathcal{U}_{\mathbf{c},\mathbf{s}}\mathfrak{g}^\theta$ admits a quasi-triangular structure by the construction of the universal $K$-matrix in \cite[Corollary 9.6]{BK}, which by \cite[Cor.\ 3.18]{Kol20} equips $\Rep_{\mathbf{c},\mathbf{s}}(G^\theta)$ with a braided module structure over $\Rep_q(G)$ (or its $\Z_2$-equivariantization in the outer case, as we will explain next).   

\subsubsection{Outer involutions and equivariantization}
\label{subsec:equivariantization}

In the case when $\theta$ is an \textit{outer} involution, \ie $\theta$ is an automorphism of $G$ which is not induced by the adjoint action of any element in $G$, a slight modification of the above recollection is required. In this case, we consider the equivariantized category $\mathrm{Rep}(G)^{\Z_2} := \mathrm{Rep}(G \rtimes \Z_2)$ with $\Z_2$ acting by $\theta$. Then $G^\theta$ embeds as a subgroup of $G \rtimes \Z_2$ via $h \mapsto (h,+)$, fixed by the conjugation action of the element $J = (e, -)$. As such, we have a $\mathrm{Rep}(G)^{\Z_2}$-module structure on $\mathrm{Rep}(G^\theta)$ by restriction as before, and \cite[Cor.\ 3.18]{Kol20} gives a braided deformation $\mathrm{Rep}_{\mathbf{c}, \mathbf{s}}(G^\theta)$ over the category $\mathrm{Rep}_q(G)^{\Z_2} := \left(\mathcal{U}_q\mathfrak{g} \rtimes \Z_2\right)\modu$ for specializable parameters.

Hence, we have the following fundamental Theorem due to \cite{Kol20} and \cite{BK}:
\begin{thm}\label{thm:Kolb-thm}
    The category $\Rep_{\mathbf{c},\mathbf{s};q}(G^\theta)$ is a braided module category over $\Rep_q(G\rtimes \langle \sigma \rangle)\cong \Rep_q(G)^{\langle\sigma\rangle}$. The group $\langle \sigma\rangle$ is trivial in the inner and $\Z_2$ in the outer case. 
\end{thm}

\subsubsection{The balancing of quantum symmetric pairs}

While Theorem~\ref{thm:Kolb-thm} does not include the data of a balancing on $\Rep_{\mathbf{c},\mathbf{s}}(G^\theta)$ with respect to its braided $(\Rep_qG)^\sigma$-module structure, it is known abstractly from \cite[Theorem\ 3.12]{BZBJ18b} that a canonical balancing exists. We present the balancing explicitly in the setting of Theorem~\ref{thm:Kolb-thm}. 

Let us first consider the Hopf-algebraic perspective of balancing structures. Let $H$ be a ribbon Hopf algebra with universal $R$-matrix $R$ and ribbon element $v$. The category $\Rep(H)$ is a ribbon tensor category with twist given by the action of $v^{-1}$, \ie $\theta_X:= v^{-1}. (-)$. If $B$ is a quasi-triangular coideal subalgebra in the sense of Definition~\ref{def:quasi-triangular-coideal} with universal $K$-matrix $K$, then an invertible element $z\in H$ defines a balancing $\varphi:\Rep(B)\xrightarrow{\sim}\Rep(B)$ via $\varphi_M:= z. (-)$ if and only if 
\begin{equation}\label{eq:balancing-element}
    z\cdot b = b\cdot z \quad \forall b \in B  \quad \text{and} \quad \Delta(z) = K\cdot(z\otimes 1) ~.
\end{equation}
Returning to the quantum symmetric pair case we observe that such an element exists by taking $z= \mathcal{J}$ as in Remark~\ref{rem:K-matrix-vs-J-matrix}. The balancing condition in \eqref{eq:balancing-element} is given in \cite[Eq.\ (3.37)]{Kol20}.

\section{Codimension two defects}
\label{sec:skeins-def}

In this section we construct skein theory in the presence of codimension two defects (Section~\ref{subsec:sk-def}), achieving one of the main goals (Theorem \ref{thm:intro-def-sk-TQFT}) of this project. We then relate the preceding construction to stratified factorization homology on surfaces in Section~\ref{subsec:fact-hom}, following the strategy of Cooke \cite{Cooke} in the case without defects.

\subsection{Skein construction with codimension two defects}
\label{subsec:sk-def}

In this section, we will construct a skein TQFT with codimension two defects for a fixed pair $(\mathcal{A},\mathcal{M})$ consisting of a ribbon tensor category $\mathcal{A}$ and a (balanced) braided module category $\mathcal{M}$ over $\mathcal{A}$. For a recent consideration of skein theory with non-trivial codimension one and codimension two defects see also \cite{BGJV}.

We will first define the source category $\tglBord$. 

\begin{definition}\label{def:tgl}
    A \textit{tangle} $T$ in an oriented $3$-manifold $M$ is a compact oriented $1$-dimensional submanifold, such that $\partial T = T\cap \partial M$.
\end{definition}
The orientation of $T$ induces an orientation of its boundary points $\partial T$, \ie a sign $\epsilon_i$ for each point $x_i \in T \cap \partial M$. The set of pairs $(x_i, \epsilon_i)$ will be denoted by $P_T$. A tangle $T$ in $M $ such that $P_T = \varnothing$ is the same as an oriented framed link embedding $L \subset M$. 

\begin{definition}\label{def:bord-tgl}
    The \textit{category of 3-bordisms with tangles} $\Bord_3^\mathrm{tgl}$ consists of:
    \begin{itemize}
        \item Objects are pairs $(\Sigma, P)$ of a compact oriented surface $\Sigma$ along with a finite set $P$ of pairs $(x_i,\epsilon_i)$ of distinct points $x_i \in \Sigma$ and signs $\epsilon_i \in \{\pm\}$. 
        \item A morphism $(\Sigma, P ) \rightarrow (\Sigma', P')$ is represented by a pair $(M,T)$ consisting of an oriented 3-bordism $M: \Sigma \rightarrow \Sigma'$ between the underlying surfaces and a tangle $T$ in $M$ such that $(\partial M, P_T)\simeq (-\Sigma \sqcup \Sigma', -P \sqcup P')$. We identify pairs $(M,T)\sim (M',T')$ whenever there is a diffeomorphism of 3-bordisms $\phi: M \simeq M'$ which satisfies $\phi(T) = T'.$
    \end{itemize}
\end{definition}
We assume positive orientation whenever any point orientations are omitted. 

Furthermore, for a set $D$ one can define the category $\Bord_3^\mathrm{tgl}(D)$ of \textit{bordisms with line defect labelled by $D$} as above where now tangles are labelled by elements of the set $D$, \ie, if $T\subset M$ is a tangle in $M$, then it carries a labelling map $f: \pi_0(T) \to D$. The identification of bordisms then reads $(M,T,f)\sim (M',T',f')$ whenever there is a diffeomorphism $\phi: M \simeq M'$ compatible with the tangles and their $D$-labellings. If $D$ consists of a single label, \ie $D=\{\ast\}$, then tautologically $\Bord_3^\mathrm{tgl}(\{\ast\}) = \Bord_3^\mathrm{tgl}$.

The category $\Bord_3^\mathrm{tgl}(D)$ is a symmetric monoidal category in the usual sense. 

\begin{definition}\label{def:TQFT-line-def}
    Let $\mathcal{S}$ be a symmetric monoidal category. A \textit{3d TQFT with line defects (labelled by $D$) and with values in $\mathcal{S}$} is a symmetric monoidal functor 
    \[\mathcal{Z}: \Bord_3^\mathrm{tgl}(D) \longrightarrow \mathcal{S}~.\]
\end{definition}

There is an obvious inclusion functor
\begin{equation}
    \iota: \Bord_3 \hookrightarrow \Bord_3^\mathrm{tgl}(D)
\end{equation}
by equipping surfaces (resp.\ bordisms) with the empty set $P= \varnothing$ (resp.\ empty tangle $T = \varnothing$). The inclusion $\iota$ is not full, since we see that
\begin{equation*}
\Hom_{\Bord_3^\mathrm{tgl}(D)}\left(\left(\Sigma,\varnothing\right), \left(\Sigma',\varnothing\right)\right)
\end{equation*}
consists of $3$-bordisms from $\Sigma\rightarrow \Sigma'$ with embedded oriented links labelled by $D$. Furthermore, we have the full, essentially surjective forgetful functor 
\begin{equation}
    F: \Bord_3^\mathrm{tgl}(D) \rightarrow \Bord_3~,
\end{equation}
which is a left inverse of $\iota$, \ie $F \circ ~\iota = \Id$. 

\begin{definition} \label{def:AM-ribbons}
Let $\AM$ be a pair of a ribbon tensor category $\A$ and a balanced braided $\A$-module category $\M$, all linear over some coefficient ring $k$. 
    \begin{itemize}
        \item An \textit{$(\mathcal{A}, \mathcal{M})$-labelling} of an object $(\Sigma, P)\in \Bord_3^\mathrm{tgl}$ is a finite collection of $\mathcal{A}$-labelled framed points in $\Sigma\setminus P$ along with $\M$-labellings and framings (tangent vectors) for $P$. 
        \item An \textit{$(\mathcal{A}, \mathcal{M})$-ribbon graph} in a 3-manifold $M$ with a tangle $T$ is a finite oriented ribbon graph $\Gamma$ in $M$ containing $T\subset \Gamma$ such that $\Gamma \setminus T$ is an $\mathcal{A}$-labelled ribbon graph, and $T$ is an $\mathcal{M}$-coloured ribbon graph (see Figure~\ref{fig:line-def-local}). 
        \item The $R$-linear category of $\AM$-ribbon graphs in $\Sigma\times [0,1]$ is denoted $\operatorname{Rib}_\AM(\Sigma,P)$ and consists of:
        \begin{itemize}
        \item $\AM$-labellings of $(\Sigma,P)$ as objects and 
        \item $\Hom_{\operatorname{Rib}_\AM(\Sigma,P)}(\mathscr{L}_0,\mathscr{L}_1):= k\langle \AM\text{-ribbon graphs }\Gamma\text{ in }(\Sigma\times [0,1], P\times [0,1])\mid \Gamma|_{\Sigma \times \{i\} = \mathscr{L}_i}\rangle$.
        \end{itemize}
    \end{itemize}
\end{definition}

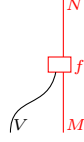
\begin{figure}
\centering
\begin{tikzpicture}[yscale=1]
    \draw[red] (1,0) node [anchor=south west][inner sep=0.75pt]  [font=\tiny]  {$M$} -- (1,0.8);
    \draw[red] (1,1) -- (1,1.8) node [anchor=north west][inner sep=0.75pt]  [font=\tiny]  {$N$};
    \draw[draw=red] (0.8,0.8) rectangle ++(0.3,0.2) node [anchor=north west][inner sep=0.75pt]  [font=\tiny, color=red]  {$f$};
    \draw (0.3,0) node [anchor=south west][inner sep=0.75pt]  [font=\tiny]  {$V$} .. controls (0.3,0.5) and (0.8,0.3) .. (0.9,0.8);
\end{tikzpicture}
\caption{The $\mathcal{A}$-ribbon graphs that connect to the $\mathcal{M}$-labelled ribbon defect are described locally by the above picture. Here, the coupon is labelled by an $\mathcal{M}$-morphism $f\in \Hom_\mathcal{M}(V\triangleright M, N)$. The source and target are determined by the orientation of the underlying tangle.}
\label{fig:line-def-local}
\end{figure}

Recall that for the disk $\mathbb{D}$ without defects, the category $\operatorname{Rib}_\AM(\mathbb{D}) = \operatorname{Rib}_\A(\mathbb{D})$ has a natural ribbon tensor categorical structure and Reshetikhin--Turaev graphical calculus provides the evaluation functor
\begin{equation}\label{eq:A-evaluation-functor}
    \mathrm{ev}_\A: \operatorname{Rib}_\A(\mathbb{D})\longrightarrow \A
\end{equation}
which is a full surjective ribbon functor. This functor is used to define the local relations
\begin{equation}\label{eq:A-local-relations}
    \mathrm{lr}_\A := \ker(\mathrm{ev}_\A)
\end{equation}
which define $\A$-skein theory without defects. 

Now we write $\mathbb{D}_*$ for the 2-disk with a defect point in the center. Then the category $\operatorname{Rib}_\AM(\mathbb{D}_\ast)$ inherits the structure of a balanced braided module category over $\operatorname{Rib}_\A(\mathbb{D})$ which we describe graphically as follows. The $\operatorname{Rib}_\mathcal{A}(\mathbb{D})$-module structure is induced by the embedding $\mathbb{D}\cup \mathbb{D}_\ast \hookrightarrow \mathbb{D}_\ast$: 
\begin{equation}
\begin{tikzpicture}[baseline={([yshift=-.5ex]current bounding box.center)}]
    \filldraw[fill=gray!30, draw=gray, dotted] (-1.5,0) circle (0.5);
    \node at (-2,-0.5) {$\mathbb{D}$};
  \filldraw[fill=gray!30, draw=gray, dotted] (0,0) circle (0.5);
  \node at (-0.4,-0.5) {$\mathbb{D}_\ast$};
  \fill[red] (0,0) circle (1.5pt);
  \node at (-0.75,0) {$ \cup$};
\end{tikzpicture}
\hookrightarrow
\begin{tikzpicture}[baseline={([yshift=-.5ex]current bounding box.center)},scale = 1]
  \filldraw[fill=gray!10, draw=gray, dotted] (0,0) circle (1);
  \filldraw[fill=gray!30, draw=gray, dotted] (-0.6,0) circle (0.25);
  \filldraw[fill=gray!30, draw=gray, dotted] (0,0) circle (0.25);
  \fill[red] (0,0) circle (1.5pt);
\end{tikzpicture}
\end{equation}
while the braided module structure is given by the isotopy
\begin{equation}\label{eq:brmod-operad}
\begin{tikzpicture}[baseline={([yshift=-.5ex]current bounding box.center)},scale = 1]
  \filldraw[fill=gray!10, draw=gray, dotted] (0,0) circle (1);
  \filldraw[fill=gray!30, draw=gray, dotted] (-0.6,0) circle (0.25);
  \filldraw[fill=gray!30, draw=gray, dotted] (0,0) circle (0.25);
  \fill[red] (0,0) circle (1.5pt);
  \draw[->, gray
  ] (-0.5,-0.35) arc[start angle=-140, end angle=140, radius=0.6];
\end{tikzpicture}
\end{equation}

The balancing is induced by the $2\pi$-rotation of $\mathbb{D}_\ast$. 

\begin{prop}\label{prop:AM-evaluation-functor}
We have a full surjective functor obtained from graphical calculus of balanced braided module categories:
\begin{equation}\label{eq:AM-evaluation-functor}
\mathrm{ev}_\AM:\operatorname{Rib}_{\mathcal{A},\mathcal{M}}(\mathbb{D}_\ast) \longrightarrow \mathcal{M}~
\end{equation}
compatible with the Reshetikhin--Turaev evaluation $\mathrm{ev}_{\A}: \mathrm{Rib}_{\A}(\mathbb{D}) \to \A$.
\end{prop}
\begin{proof}
Consider the skeleton subcategory $\operatorname{Rib}_\AM \subset \operatorname{Rib}_\AM(\mathbb{D}_\ast)$ defined as follows: An object in $\operatorname{Rib}_\AM$ is an $\AM$-labelling with all $\A$-points on the positive $x$-axis and all framings (including for the $\M$-point) in the positive $x$-direction. The equivalence $\operatorname{Rib}_\AM\simeq \operatorname{Rib}_\AM(\mathbb{D}_\ast)$ follows from the choosing any isotopy from an arbitrary $\AM$-labelling on $\mathbb{D}_\ast$ to a representative $\AM$-labelling on the positive $x$-axis. 

Therefore, $\mathrm{ev}_\AM$ is determined by its restriction 
\[\mathrm{ev}_\AM: \operatorname{Rib}_\AM \to \M\]
which takes a representative $[m,x_1,\dots, x_n]$ to $m\ract x_1\otimes\dots\otimes x_n$ and an $\AM$-ribbon graph to the associated evaluated morphism in $\M$. 
\end{proof}
The kernel of this functor describes local relations near the $\M$-labelled defect and is denoted 
\begin{equation}
    \label{eq:AM-local-relations}
    \mathrm{lr}_\AM:= \ker\left(\mathrm{ev}_\AM\right) ~.
\end{equation}

\begin{definition}\label{def:skmod}
    Let $M$ be an oriented 3-manifold with a tangle $T$ and let $B$ be an $\AM$-labelling of $(\partial M,P_T)$. The \textit{defect} $\AM$-\textit{skein module} is defined as
    \[
    \sk_{\AM}(M,T;B) := k \langle\AM\text{-ribbon graphs in }(M,T)\rangle/(\mathrm{lr}_\AM, \mathrm{lr}_\A)~
    \]
    where we only consider labellings compatible with the boundary labelling $B$. 
\end{definition}
In particular, for the regular module $\mathcal{M} = \mathcal{A}$ we get the underlying $\mathcal{A}$-skein module of $M$. 

\begin{definition}\label{def:skcat}
    Let $\Sigma \in \Bord^\mathrm{tgl}_3$ be an oriented surface with point defects. The \textit{defect} $(\mathcal{A},\mathcal{M})$-\textit{skein category} $\sk_{(\mathcal{A},\mathcal{M})}(\Sigma)$ consists of:
    \begin{itemize}
        \item Objects: $(\mathcal{A},\mathcal{M})$-labellings of $\Sigma$ and
        \item Morphisms: $\Hom_{\sk_{(\mathcal{A},\mathcal{M})}(\Sigma)}(B,B') = \sk_{(\mathcal{A},\mathcal{M})}(\Sigma\times I; B \sqcup B')$~.
    \end{itemize}
    Composition is given by stacking defect skeins in the interval direction. 
\end{definition} 
Combining Definitions~\ref{def:skcat} and \ref{def:skmod}, a bordism $M:\Sigma_\mathrm{in}\rightarrow \Sigma_\mathrm{out}$ in $\Bord_3^\mathrm{tgl}$ gives rise to a profunctor:
\begin{equation}
\label{eq:skmod-profunctor}
\sk_{\AM}(M):\sk_{\AM}(\Sigma_\mathrm{in})\proarrow \sk_{\AM}(\Sigma_\mathrm{out})
\end{equation}

The above constructions realize skein theory with line defects as a defect TQFT, in the sense of Definition \ref{def:TQFT-line-def}, with values in the bicategory $\Bimod$ (Definition \ref{def:bimod}). 
\begin{thm}\label{thm:def-sk-TQFT}
The assignment $\Sigma\mapsto \sk_\AM(\Sigma)$ and Equation~\eqref{eq:skmod-profunctor} form a symmetric monoidal functor 
\begin{equation}\label{eq:sk-TQFT}
\sk_\AM: \Bord_3^\mathrm{tgl} \longrightarrow \Bimod~.
\end{equation}
\end{thm}

\begin{rem}\label{rem:defects-from-multiple-brmods}
    So far, we have restricted our attention to a single defect label $D= \{\M\}$ of a single braided module category. It is straightforward to include multiple such labels, \ie $D = \{\M_i\}_{i\in I}$ a family of braided module categories over $\A$. This will result into a skein theory with line defects forming a 3d TQFT
    \[\sk_{(\A,\{\M_i\})}: \Bord_3^\mathrm{tgl}(D) \longrightarrow \Bimod~.\]
\end{rem}

An immediate application of Theorem~\ref{thm:def-sk-TQFT} yields
\begin{cor}\label{cor:HH-skcat}
Let $\Sigma$ be a defect surface (\ie a surface with marked points). Considering the defect 3-manifold $\Sigma \times \mathbb{S}^1$ (where the defect links are the embedded circles traced out by the defect points), we have
\[
\sk_{\AM}(\Sigma\times \mathbb{S}^1) = \operatorname{HH}_0(\sk_\AM(\Sigma))
~.\]
That is, dimensional reduction on $\mathbb{S}^1$ corresponds to taking the $0$th Hochschild homology
\end{cor}
While we do not continue along this train of thought as it would distract from our main purposes, we observe that similar strategies as in \cite{BJVV} may be employed to compute defect skein module dimensions using the preceding Corollary.

\subsection{Relation to factorization homology}
\label{subsec:fact-hom}

Just as the $\A$-skein category may be identified with the \textit{a priori} more universal construction of \textit{factorization homology} over a surface \cite[Theorem 4.3.1]{Cooke}, we show that the defect skein category can be computed via \textit{stratified factorization homology} \cite{AFT17,BZBJ18b} as well. 

Recall that factorization homology with coefficients $\AM$, denoted
$$\int_{\AM}: \mathrm{Man}_2^{\mathrm{or,mrk}} \longrightarrow \Bimod$$
is defined as the left Kan extension: 
\begin{equation}\label{eq:str-fact-hom}
\begin{tikzcd}
\operatorname{Disk}_2^\mathrm{or,mrk}\arrow[d, hook]\arrow[r, "\AM"] & \Bimod\\ 
\operatorname{Man}_2^\mathrm{or,mrk}\arrow[ur, dashed,swap, "\int_{\AM}"] 
\end{tikzcd}
\end{equation}
where $\mathrm{Disk}_2^{\mathrm{or,mrk}}$ (resp. $\mathrm{Man}_2^{\mathrm{or,mrk}}$) denotes the category whose objects are disjoint unions of $\mathbb{D}$ and $\mathbb{D}_*$ with oriented stratified embeddings (resp. oriented marked surfaces with oriented stratified embeddings).
\begin{prop}\label{prop:sk-cat-fact}    
    The defect $(\mathcal{A},\mathcal{M})$-skein category is computed by stratified factorization homology: if $(\Sigma,P)\in \Bord_3^\mathrm{tgl}$ is a surface with point defects, then
    \[
    \sk_{\AM} (\Sigma,P)\simeq \int_{(\Sigma,P)}\AM~.
    \]
\end{prop}
\begin{proof}
Stratified factorization homology $(\Sigma,P) \mapsto \int_{(\Sigma,P)}(\mathcal{A},\mathcal{M})$ is uniquely characterized by the following properties (as a consequence of its presentation as a left Kan extension in \cite[Definition 2.5]{BZBJ18b}):
\begin{enumerate}
    \item For surfaces $\Sigma$ without defects, defect factorization homology with coefficients in $(\mathcal{A},\mathcal{M})$ coincides with usual factorization homology with coefficients in $\mathcal{A}$:
    $$\int_\Sigma \, (\mathcal{A},\mathcal{M}) \simeq \int_\Sigma \mathcal{A} \quad;$$
    \item for the 2-disk $\mathbb{D}_*$ with defect at the origin, we have
    $$\int_{\mathbb{D}_*}(\mathcal{A},\mathcal{M}) \simeq \mathcal{M} \quad;$$
    \item $\int_{-}(\mathcal{A},\mathcal{M})$ satisfies excision: for $(\Sigma,P) = (\Sigma_1, P_1) \, \cup_{\Ann} \, (\Sigma_2, P_2)$ a gluing of two marked surfaces over an unmarked annulus $\Ann$, then
    $$\int_{(\Sigma,P)}(\mathcal{A},\mathcal{M}) \simeq \int_{(\Sigma_1,P_1)}(\mathcal{A},\mathcal{M}) \, \otimes_{\int_{\Ann}\mathcal{A}} \, \int_{(\Sigma_2,P_2)}(\mathcal{A},\mathcal{M}).$$
\end{enumerate}
Condition (1) reduces to the fact that the $\AM$-skein category of a surface $\Sigma$ with no defects is equivalent to the $\A$-skein category of $\Sigma$, the latter of which has been identified with factorization homology $\int_\Sigma \A$ \cite[Theorem 4.3.1]{Cooke}.
Note that condition (3) along with $\int_{\mathbb{D}}{\AM}\simeq \A$ imply condition (1). 

For (2), note that Proposition \ref{eq:AM-evaluation-functor} and Definition \ref{def:skmod} computes the $\AM$-skein category of the defect disk $\mathbb{D}_*$ as   
\[\sk_\AM(\mathbb{D}_\ast):= \operatorname{Rib}_\AM(\mathbb{D}_\ast)/\ker(\mathrm{ev}_\AM) \simeq \M~.\]

Finally for (3), we may reduce to showing excision around a single defect point, \ie let $\Sigma$ be a connected surface with a point defect at $x \in \Sigma$, as the general case is similar. We shall establish that
\begin{equation} \label{eq:excision-of-one-point-defect}
    \sk_{\AM}(\Sigma,\{x\}) \simeq \sk_\A(\Sigma\setminus \mathbb{D}_x)\boxtimes_{\sk_\A(\Ann)}\M.
\end{equation}
where $\mathbb{D}_x$ is small disk around $x$. The proof is modelled on \cite[Theorem 4.2.16]{Cooke} so we only provide a sketch, emphasizing at times some key features in our setting. 

To this end, we introduce some topological terminology following \textit{loc. cit}, specializing to the cases we need. Let $\Sigma_L := \Sigma \setminus \mathbb{D}_x$, and let $\Sigma_R := \mathbb{D}_x$.
\begin{itemize}
    \item \cite[Definition 4.2.11]{Cooke}  A \textit{right (resp. left) thickened embedding} of $\mathbb{S}^1$ into $\Sigma_R$ (resp. $\Sigma_L$)  is a triple $(\Xi, E, \lambda)$ where (i) $\Xi: \mathbb{S}^1 \times (-\varepsilon,1] \hookrightarrow \Sigma_R$ (resp. $\Xi: \mathbb{S}^1 \times [0,1+\varepsilon) \hookrightarrow \Sigma_L$) such that $\Xi_1$ (resp. $\Xi_0$) gives an $\mathbb{S}^1$-parametrization of $\partial \Sigma_R$ (resp. $\partial \Sigma_L$), (ii) $E: M \hookrightarrow M$ is an embedding disjoint from the image of $\Xi$, and (iii) $\lambda: \mathrm{Id}_M \to E$ is an isotopy supported only on the image of $\Xi$.
    \item \cite[Remark 4.12]{Cooke} Given a left thickened embedding $(\Xi, E, \lambda)$ of $\mathbb{S}^1$ into $\Sigma_L$ and any object $a \in \sk_{\A}(\Sigma_L)$, the isotopy $\lambda$ traces out a ribbon tangle $r_{\lambda,a}: a \to E(a)$, such that for any morphism $f: a \to a'$ in $\sk_{\A}(\Sigma_L)$ the following diagram commutes
    \begin{equation}
        \begin{tikzcd}
	{a'} & {E(a')} \\
	a & {E(a)}
	\arrow["{r_{\lambda, a'}}", from=1-1, to=1-2]
	\arrow["f", from=2-1, to=1-1]
	\arrow["{r_{\lambda,a}}", from=2-1, to=2-2]
	\arrow["{E(f)}"', from=2-2, to=1-2]
\end{tikzcd}
    \end{equation}
    Given a right thickened embedding $(\Xi,E,\lambda)$ of $\mathbb{S}^1$ into $\Sigma_R$, analogous ribbon tangles can be constructed for $\sk_{(\A,\M)}(\Sigma_R) \simeq \M$, where we note that $\lambda$ is a trivial isotopy (and hence $E$ is the identity map) in a neighborhood of the center point $x \in \Sigma_R$.
    \item \cite[Definition 4.2.17]{Cooke} From now on, we fix a left thickened embedding $(\Xi_L, E_L,\lambda_L)$ of $\mathbb{S}^1$ into $\Sigma_L$. Let $m \in \sk_{\A}(\Sigma_L)$ and let $a,b \in \sk_{\A}(\Ann)$ where $a$ and $b$ have disjoint support in the annulus. Then our left thickened embedding defines a ribbon graph
    \begin{equation}
        \rho_{m;a,b}:= r_{\lambda_L,m \ract a} \sqcup \mathrm{Id}_{\varnothing \ract b}: m \ract (a \sqcup b) \longrightarrow (m \ract a) \ract b.
    \end{equation}
    Similarly, from now on we fix a right thickened embedding $(\Xi_R, E_R,\lambda_R)$ of $\mathbb{S}^1$ into $\Sigma_R$. Then for every object $n \in \sk_{(\A,\M)}(\Sigma_R) \simeq \M$ and $a,b \in \sk_{\A}(\Ann)$ as above, our right thickened embedding defines a ribbon graph
    \begin{equation}
        \rho_{a,b;n}:= r_{\lambda_R, b \lact n} \sqcup \mathrm{Id}_{a \lact \varnothing}: (a \sqcup b) \lact n \longrightarrow a \lact (b \lact n).
    \end{equation}
\end{itemize}
With these topological definitions established, we can proceed to establish \eqref{eq:excision-of-one-point-defect} by building a functor $F$ going from right to left:
\begin{itemize}
    \item Given an object $(m,n)$ in $\sk_\A(\Sigma_L)\boxtimes_{\sk_\A(\Ann)}\M$, we define
    \begin{equation}
        F(m,n) := E_L(m) \sqcup E_R(n).
    \end{equation}
    \item By the definition of the relative tensor product over $\sk_{\A}(\Ann)$, the morphisms on the right hand side are generated by the following two types of morphisms: (i) $(f,g)$ where $f$ is a morphism in $\sk_{\A}(\Sigma_L)$ and $g$ is a morphism in $\sk_{(\A,\M)}(\Sigma_R)$, and (ii) invertible morphisms $\iota_{(m,a,n)}:(m \ract a,n) \to (m, a \lact n)$ for $(m,a,n) \in \sk_\A(\Sigma_L) \times \sk_\A(\Ann) \times \sk_{(\A,\M)}(\Sigma_R)$ (and their inverses, which are treated similarly). We define $F$ on $(f,g)$ by
    \begin{equation}
        F(f,g) := E_L(f) \sqcup E_R(g).
    \end{equation}
    We define $F$ on $\iota_{m,a,n}$ by 
    \begin{equation}
        F(\iota_{(m,a,n)}) := r^{-1}_{\lambda_L,E_L^2(m)} \, \sqcup \, r_{\lambda_R,a} \circ r^{-1}_{\lambda_L,E_L(a)} \, \sqcup \, r_{\lambda_R,n}
    \end{equation}
    as a morphism in $\sk_{(\A,\M)}(\Sigma_L \sqcup_{\Ann} \Sigma_R)$ between
    $$E^2_L(m) \sqcup E_L(a) \sqcup E_R(n) \longrightarrow E_L(m) \sqcup E_R(a) \sqcup E_R^2(n).$$
\end{itemize}
The fact that the preceding description of $F$ is well-defined, and that $F$ is essentially surjective and fully faithful follow verbatim the arguments of \cite[Theorem 4.2.16]{Cooke}. Indeed, the most subtle part is the full and faithfulness of $F$ for ribbon tangles in the annular region, where the arguments of \textit{loc. cit} applies without change.
\end{proof}

\section{Internal skeins with defects}
\label{sec:IS}

In this section, we will describe internal skeins with line defects extending the setting of \cite{GJS,JR}.

\subsection{Internal skein algebras}\label{subsec:ISA}

We aim to apply the monadic reconstruction results of Proposition~\ref{prop:braided-module-cat-reconstruction} to reconstruct defect skein categories. Let $\Sigma\in\Bord_3^\mathrm{tgl}$ be a connected surface with point defects and fix a disk embedding $\mathbb{D}\hookrightarrow \Sigma$ in $\mathrm{Man}_2^{\mathrm{or,mrk}}$, in particular avoiding all defect points. Let $\pSigma$ be the associated punctured surface, whose puncture induces a braided $\A$-module category structure on $\sk_\AM(\pSigma)$.

\begin{defprop}\label{defprop}
Let $\M$ be a braided $\A$-module category and $\unit_\M \in \M$ be an $\A$-progenerator. The \textit{internal (defect) skein algebra} of $\pSigma$ is defined as
\[\skalg^{\mathrm{int}}_{\AM}(\pSigma):= \underline{\End}(\unit_{\pSigma}) \in \Ahat\]
where $\unit_{\pSigma}$ denotes the $\AM$-labelling of $\pSigma$ with all defects labelled by $\unit_\M$ and no $\A$-labelled points. 

In particular, we have an equivalence of $\widehat\sk_{\A}(\Ann)$-categories: 
\[
\sk_{\AM}(\pSigma) \simeq \skalg^{\mathrm{int}}_{\AM}(\pSigma)\modu(\widehat\sk_\A(\Ann))~.
\]
\end{defprop}

\begin{rem}\label{rem:ISA-disconnected-and-multiple-gates}
\begin{enumerate}
    \item If the surface $\Sigma$ above were disconnected, then we need to puncture each connected component of the surface once and consider an $\A^{\boxtimes m}$-action where $m=|\pi_0(\Sigma)|$. Otherwise, the progeneration hypothesis for monadic reconstruction will not be satisfied.
    \item One can also puncture a connected surface $\Sigma$ several times, say $k$ times, and consider the associated $\A^{\boxtimes k}$-action. The resulting internal skein algebra is one with $k$ gates (see \cite{JR}) and lives in $\Ahat^{\boxtimes k}$. 
\end{enumerate}
\end{rem}

\begin{example}\label{eg:connected-sum-skcat}
    For a fixed connected surface $S$, let $\M = \sk_\A(S^\ast)$ be the associated braided $\A$-module category. Let $\Sigma$ be a connected surface with a single $\M$-labelled defect. Then we have 
    \[\sk_\AM(\Sigma) \simeq \sk_\A(\Sigma\#S)\]
    and 
    \[\sk^{\mathrm{int}}_\AM(\pSigma) \simeq \sk^{\mathrm{int}}_\A(\pSigma\#S)~.\]
    In particular, $S= \mathbb{S}^2$ corresponds to the regular module $\M=\A$ and thus the transparent defect. 
\end{example}

Let $\Sigma_{g,n}^b$ denote the genus $g$ surface with $n$ defect points and $b$ boundary components. We also fix one boundary component to carry the $\Ann$-module structure. We introduce the following notation: 
\begin{equation}\label{eq:ISA-notation}
\sA_{\pSigma}:= \skalg^{\mathrm{int}}_{\AM}(\pSigma)\quad\text{and} \quad \sA_{g,n}^b := \skalg^{\mathrm{int}}_{\AM}(\Sigma_{g,n}^b)~.
\end{equation}
In particular, $\sA_{g,0}^1$ is the internal skein algebra of $\Sigma_g^\ast$ with no defects as in \cite{GJS}. We write in particular
\begin{align*}
    & \sA_{\Ann} = \sA_{0,0}^2 \simeq \Ocal_\A 
    &\sA_{\mathbb{D}_\ast} = \sA_{0,1}^1 \simeq \iend(\unit_\mathcal{M}) \\
    & \sA_{\ptorus} = \sA_{1,0}^1 =: \Dcal_\A
    &
\end{align*}
or in pictures: 
\begin{equation}
\begin{tikzpicture}[baseline={([yshift=-.5ex]current bounding box.center)}]
  \filldraw[fill=gray!30, draw=gray,dotted] (0,0) circle (0.5); 
  \fill[fill = white, draw=gray,dotted] (0,0) circle (0.2);
\end{tikzpicture}~\overset{\sA}{\mapsto} \Ocal_\A, \quad 
\begin{tikzpicture}[baseline={([yshift=-.5ex]current bounding box.center)}]
  \filldraw[fill=gray!30, draw=gray] (0,-0.4) .. controls (-0.1,-0.4) and (-1,-0.8) .. (-1,-0.2) .. controls (-1,0.6) and (-0.1,0) .. (0,0); 
  \filldraw[fill=gray!20,draw=gray, dotted] (0,0) arc (90:270:0.1 and 0.2);
  \filldraw[fill=gray!20, draw=gray, dotted] (0,-0.4) arc (-90:90:0.1 and 0.2);
  \filldraw[fill=white, draw=gray] (-0.7,-0.2) .. controls (-0.6,-0.3) and (-0.5,-0.3) .. (-0.4,-0.2)
  .. controls (-0.5,-0.1) and (-0.6,-0.1) .. (-0.7,-0.2);
\end{tikzpicture}\overset{\sA}{\mapsto} \Dcal_\A, \quad
\begin{tikzpicture}[baseline={([yshift=-.5ex]current bounding box.center)}]
  \filldraw[fill=gray!30, draw=gray, dotted] (0,0) circle (0.5);
  \fill[red] (0,0) circle (1.5pt);
\end{tikzpicture}\overset{\sA}{\mapsto} \iend(\unit_\M).
\end{equation}
The following proposition is a generalization of \cite[Cor.\ 3.4]{JR} by including point defects:
\begin{prop}\label{prop:ISA-decomp}
    The internal skein algebra $\sA_{g,n}^b$ admits the following tensor decomposition 
    \[ \sA_{g,n}^b \simeq \Dcal_\A^{\totimes g}\totimes \Ocal_\A^{\totimes b-1} \totimes \iend(\unit_\M)^{\totimes n}\]
\end{prop}
\begin{proof}
The proof follows verbatim as in the case without defects by application of \cite[Theorem 3.2]{JR} or \cite[Theorem\ 5.14]{BZBJ18a} to the handle and comb decomposition in Figure~\ref{fig:hcomb-w-defects}.
\end{proof}
\begin{rem}\label{rem:ISA-decomp-multiple-gates}
    As in \cite{JR}, one can include multiple gates instead of a single gate and present a generalized version of Proposition~\ref{prop:ISA-decomp}. 
\end{rem}

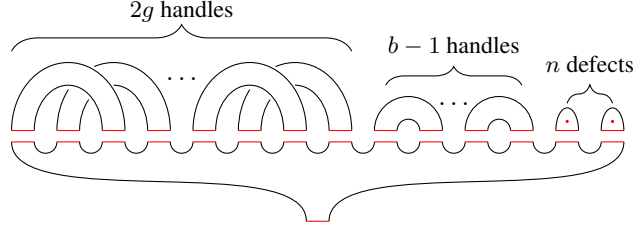
\begin{figure}
    \centering
    \begin{tikzpicture}[scale=0.3]
\draw (0,3.5) to[out=270,in=90,looseness=0.6] (13,0);
\draw (27,3.5) to[out=270,in=90,looseness=0.6] (14,0);
\draw[red] (13,0) -- (14,0);
\foreach \i in {0,8} {
\draw (\i,4) arc (180:0:2.5 and 3);
\draw (\i+1,4) arc (180:0:1.5 and 2);
\draw (\i+2,4) arc (180:143:2.5 and 3);
\draw (\i+7,4) arc (0:110:2.5 and 3);
\draw (\i+3,4) arc (180:137:1.5 and 2);
\draw (\i+6,4) arc (0:90:1.5 and 2);
}
\foreach \i in {16,20} {
\draw (\i,4) arc (180:0:1.5 and 1.5);
\draw (\i+1,4) arc (180:0:0.5 and 0.5);
}
\foreach \i in {0,2,...,26} {
  \draw[red] (\i,4) -- (\i+1,4);  
  \draw[red] (\i,3.5) -- (\i+1,3.5);  
  \ifthenelse{\i <26}{\draw (\i+1,3.5) arc (-180:0:0.5 and 0.5)}{};
}
\foreach \i in {24,26}{
\draw (\i,4) arc (180:0:0.5 and 1);
\fill[red] (\i+0.5,4.4) circle (2pt);
}
\node at (7.6,6.3) {$\dots$};
\draw [decorate,decoration={brace,amplitude=10pt}]
  (0,7.5) -- (15,7.5) node[midway,yshift=1.5em]{\small $2g$ handles};
  \node at (19.6,5.2) {$\dots$};
\draw [decorate,decoration={brace,amplitude=7pt}]
  (16.5,6) -- (22.5,6) node[midway,yshift=1.5em]{\small $b-1$ handles};
  \draw [decorate,decoration={brace,amplitude=7pt}]
  (24.5,5.2) -- (26.5,5.2) node[midway,yshift=1.5em]{\small $n$ defects};
\end{tikzpicture}
    \caption{The handle and comb presentation of the surface $\Sigma_{g,n}^b$.}
    \label{fig:hcomb-w-defects}
\end{figure}

\subsection{Internal skein modules}
\label{subsec:ISM}
Thus far, we have considered internal skeins on surfaces where the story runs in parallel to \cite{BZBJ18b} after passing through the equivalence of Proposition~\ref{prop:sk-cat-fact}. We now continue our story in three dimensions by introducing \textit{internal skein modules} of $3$-manifolds with line defects. Throughout this section, $\M$ will be a braided module category over $\A$ with $\A$-progenerator $\unit_\M\in \M$ and 
\[F:= \mathrm{act}_{\unit_\M}:\A\to \M\]
will denote the action functor on $\unit_\M$. 

Let $M: \varnothing \rightarrow \Sigma$ be a morphism in $\Bord_3^{\mathrm{tgl}}$. For simplicity, let us assume that both $M$ and $\Sigma$ are connected\footnote{In the disconnected case, we can repeat the same procedure by removing a disk for each connected component.}. The defect skein module of $M$ (see Definition~\ref{def:skmod} and \eqref{eq:skmod-profunctor}) defines a presheaf
\begin{equation*}
    \sk_{\AM}(M):\sk_\AM(\Sigma)^\mathrm{op} \longrightarrow \sk_\AM(\varnothing) \simeq\Mod_{k}~.
\end{equation*}
Fixing a disk $\mathbb{D} \subset \Sigma$ and pulling back the presheaf along $\Sigma^\ast\hookrightarrow \Sigma$ yields: 
\begin{equation*}
    \sk_{\AM}(M):\sk_\AM(\Sigma^\ast)^\mathrm{op} \longrightarrow \Mod_{k}~.
\end{equation*}
\begin{definition}\label{def:ISM-defect}
The \textit{internal skein module} $\sk^{\mathrm{int}}_\AM(M)$ is the $\sA_{\pSigma}$-module defined as the image of $\sk_{\AM}(M)\in \widehat{\sk}_\AM(\pSigma)$ under the equivalence of Definition/Proposition~\ref{defprop}.
\end{definition}
Similar to the non-defect case, internal skein modules are compatible with gluing according to the following Theorem. Using the TQFT property \eqref{eq:sk-TQFT} the proof follows verbatim as in \cite[Theorem\ 4.1]{GJS}.  
\begin{thm}\label{thm:int-sk-gluing}
Let $M=M_2\circ M_1$ with $M_1: \varnothing \rightarrow \Sigma$ and $M_2: \Sigma\rightarrow \varnothing$ bordisms in $\Bord_3^\mathrm{tgl}$, and suppose all spaces are connected. Then, 
\[\sk_\AM(M) \simeq \Hom_{\Ahat}\left(\unit, \sk^{\mathrm{int}}_\AM(M_2)\otimes_{\sA_{\pSigma}} \sk^{\mathrm{int}}_\AM(M_1)\right)~.\]
\end{thm}
Analogously to \cite[Cor.\ 4.2]{GJS}, we may deduce the following tensor product decomposition of defect skein modules. 
\begin{cor}\label{cor:int-sk-gluing}
    If $\Ahat$ has trivial M\"uger centre, then 
    \[\sk_{\AM}(M) \simeq \sk^{\mathrm{int}}_\AM(M_2)\otimes_{\sA_{\pSigma}} \sk^{\mathrm{int}}_\AM(M_1)~.\]
\end{cor}

\begin{rem}\label{rem:disconnected-gluing}
Theorem~\ref{thm:int-sk-gluing} and Corollary~\ref{cor:int-sk-gluing} can be easily extended to bordisms $M_1:\Sigma_1\rightarrow \Sigma$ and $M_2: \Sigma \rightarrow \Sigma_2$ in $\Bord_3^\mathrm{tgl}$ with no connectedness assumption as in \cite[Prop.\ 3.12 \& Cor.\ 3.13]{JR} using the notion of \textit{internal skein bimodules} and puncturing all connected components of $\Sigma$ at least once. However, for the sake of legibility and for interests of this paper we settle for the connected case where $\Sigma_1 = \Sigma_2=\varnothing$ as presented above. 
\end{rem}
We will now focus on defect skein modules of 3-manifolds with a single defect knot. 
\begin{definition}\label{def:ISM-defect-solid-torus}
    Let $\M$ be a braided module category over $\A$. We write 
    $$\mathcal{L}_\M := \sk^{\mathrm{int}}_{(\A,\M)}(\mathbb{D}_* \times \mathbb{S}^1)$$
    for the internal skein module of the solid torus with defect core, regarded as a $\mathcal{D}_{\A}$-module. 
\end{definition}
Applying Corollary~\ref{cor:int-sk-gluing} to a 3-manifold with a defect knot embedding yields the following useful formula.
\begin{cor}\label{cor:def-skmod}
Let $M$ be a 3-manifold with an embedded defect knot $K\subset M$ and suppose $\Ahat$ has trivial M\"uger centre. Then, \[\sk_\AM(M,K) \simeq \sk^{\mathrm{int}}_\A(M\setminus K)\otimes_{\Dcal_\A} \mathcal{L}_\mathcal{M}~.\]
\end{cor}
Therefore, in order to compute the defect skein module $\sk_{\AM}(M,K)$ we need to understand $\mathcal{L}_\M$, the internal skein module of the solid torus with defect core.

\begin{prop} \label{prop:LM-formula}
The $\mathcal{D}_\A$-module $\mathcal{L}_\M$ can be expressed as a coend
\begin{equation*}
    \mathcal{L}_\M \simeq \int^{x \in \im(F)}{x\otimes \iend(\unit_\M)\otimes x^\ast}~
\end{equation*}
over the full image category\footnote{The full image $\im(F)$ has objects $\mathrm{ob}(\im(F)):=\mathrm{ob}(\A)$ and morphism spaces $\Hom_{\im(F)}(x,y):= \Hom_{\M}(F(x),F(y))$. Equivalently, it is the full subcategory of $\M$ with objects $F(\mathrm{ob}(\A))$.} $\im(F)$ of $F: \A \to \M$. 
\end{prop}
\begin{proof}
Since $\sk_\AM$ forms a TQFT, we can compute $\mathcal{L}_\M: \A^\opp \rightarrow \Vect$ by composing the bimodules 
\[
\begin{tikzpicture}[scale=.3,baseline={([yshift=-.5ex]current bounding box.center)}]
    \draw (-3,0) arc(180:0:3 and 3);
    \draw (-1,0) arc(180:0:1 and 1);
    \draw[oriented, red] (-2,0) arc(180:0:2 and 2);
    \draw (-3,0) arc (-180:0:1 and .5);
    \draw[dotted] (-3,0) arc (180:0:1 and .5);
    \draw (1,0) arc (-180:0:1 and .5);
    \draw[dotted] (1,0) arc (180:0:1 and .5);
\end{tikzpicture}:\M\boxtimes \M^\opp \to \Vect~; m\boxtimes n \mapsto \Hom_\M(n,m)\]
and 
\[\begin{tikzpicture}[scale=.3,baseline={([yshift=-.5ex]current bounding box.center)}]
    \draw (-1,0) arc(-180:0:1 and 1);
    \draw[oriented, red] (2,0) arc(0:-180:2 and 2);
    \draw (-3,0) to[out=-90,in=90,looseness=.6] (-1,-3);
    \draw (3,0) to[out=270,in=90,looseness=0.6] (1,-3);
    \draw (-1,-3) arc (-180:0:1 and .5);
    \draw[dotted] (-1,-3) arc (180:0:1 and .5);
    \draw (-3,0) arc (-180:0:1 and .5);
    \draw (-3,0) arc (180:0:1 and .5);
    \draw (1,0) arc (-180:0:1 and .5);
    \draw (1,0) arc (180:0:1 and .5);
\end{tikzpicture}: \A^\opp\boxtimes\M^\opp \boxtimes \M \to \Vect ~; x\boxtimes m\boxtimes n \mapsto \Hom_\M(x\lact m,n)~.\]
Using the TQFT property \eqref{eq:sk-TQFT}, we can compute $\mathcal{L}_{\M}$ as a vector space
\begin{align*}
    \mathcal{L}_\M &= \int^{x\in \A, m,n\in \M}{\Hom_{\M}(m,n)\otimes \Hom_{\M}(x\otimes n, m) \otimes x^\ast}\\ 
    &\simeq \int^{x\in \A,n\in \M}{\Hom_{\M}(x\otimes n, n) \otimes x^\ast}, \text{ integrating out the }m \text{ variable}\\ 
    &\simeq  \int^{x\in \A,n\in \M}{\Hom_{\A}(x,\iend(n))\otimes x^\ast}, \text{ by adjunction}\\ 
    &\simeq \int^{n\in \M}{\iend(n)}, \text{ integrating out the }x\text{ variable} \\ 
    &\overset{(1)}{\simeq} \int^{x \in \im(F)}{\iend(x\lact \unit_\M)}~\overset{(2)}{\simeq} \int^{x \in \im(F)} \, {x\otimes \iend(\unit_\M)\otimes x^\ast}
\end{align*}
where (1) uses the fact that $\unit_\M$ is a progenerator to reduce the coend to the full image category $\im(F)$. In particular, this is the coend of $\underline{\End}(x\lact \unit_\M)$  over $x \in \A$ dinatural in $\M$-morphisms $F(x):= x \lact \unit_\M\to F(y):= y \lact \unit_\M$.  Equation (2) makes use of the canonical isomorphism $\underline{\Hom}(x\lact m, y\lact n) \simeq y \otimes \underline{\Hom}(m,n)\otimes x^\ast$ obtained from the internal Hom adjunction and rigidity. 
\end{proof}

\begin{example} 
    Taking the transparent defect which corresponds to $\M = \A$ we have $F= \id$ and $\iend(\unit_\A) \cong \unit_\A$. Therefore, $\mathcal{L}_\A \cong \Ocal_\A$ with the canonical $\Dcal_\A$-module structure via the longitude. Taking instead $\M = \int_{\mathrm{Ann}}\A$ we retrieve $\mathcal{L}_\A\cong \Dcal_\A$ as the regular $\Dcal_\A$-module.
\end{example}

\section{Skein theory with quantum symmetric pair defects}\label{sec:QSP-def}

In this section, we specialize to our coefficient categories of primary interest: we set
$$\A = \mathrm{Rep}_q(G) \,  \text{(or its equivariantized version $\mathrm{Rep}_q(G)^{\Z_2}$)}$$
and we will consider braided module categories $\M = \M_{\mathbf{c}, \mathbf{s}}$ over it associated to a quantum symmetric pair $(G,G^\theta)$ as in Section~\ref{subsec:QSP}. We remind the reader that we will only work with specializable parameters $\mathbf{c}, \mathbf{s} \in R$. 

Note that any involution $\theta$ can be written as an inner involution composed with an outer involution, \ie $\theta = \mathrm{Ad}_J\circ \phi$ for some element $J$ with $J\phi(J)$ central in $G$, and $\phi\in \mathrm{Out}(G)$ a diagram involution. To uniformize notation, we set
$$\widetilde{G} = G \rtimes \Z_2$$
in the outer case, and we simply set $\widetilde{G} = G$ in the inner case. We write $\widetilde{J} = (J, \phi) \in \widetilde{G}$ in the outer case (and $\widetilde{J} = J$ in the inner case), so that we have a $\widetilde{G}$-equivariant map
\begin{equation}\label{eq:embedding}
    G^\theta \backslash \widetilde{G}  \longrightarrow \widetilde{G}^\theta \backslash \widetilde{G} \simeq G^\theta \backslash G \longhookrightarrow \widetilde{G}
\end{equation}
$$g \longmapsto g^{-1}\widetilde{J} g$$
generalizing the $G$-equivariant closed embedding $G^\theta \backslash G \hookrightarrow G$ by $g \mapsto g^{-1}Jg$ in the inner case. Note that the centralizer of $\widetilde{J}$ in $\widetilde{G}$ may contain $G^\theta$ properly as an index 2 subgroup: indeed, if $\phi(J) = J$, then the centralizer of $\widetilde{J}$ in $\widetilde{G}$ is
$$\widetilde{G}^\theta = \langle(g,1), (g,\phi): g \in G^\theta\rangle$$
containing $G^\theta$ as an index two subgroup. This is the case for the symmetric space $(\mathrm{GL}_n, \mathrm{SO}_n)$ of Type AI, for instance. To uniformize notation, we will set $\widetilde{G}^\theta = G^\theta$ in the case when $\phi(J) \neq J$ as well. We note however that in any case, the morphism \eqref{eq:embedding} is a \textit{finite map}, a closed embedding precomposed with a (trivial) double cover in the outer case.

\subsection{The quantum moment map of quantum symmetric pairs}
\label{subsec:qmm-QSP}

Abstractly, whenever $\M$ is a braided module category over $\A$ generated by some $\mathbf{1}_\M \in \M$, we know by the main theorem of \cite{BZBJ18b} that there is a quantum moment map 
\begin{equation}\label{eq:qmm}
    \mu: \mathcal{O}_{\A} = \Oq(\widetilde{G}) \longrightarrow \iend(\mathbf{1}_{\M})~.
\end{equation}
induced by and equivalent to the braided module structure. In the case of quantum symmetric pairs, we have $\mathbf{1}_{\M} = \C$ is the trivial representation, and we may understand $\mu$ explicitly as acting by $\mathcal{J}$ via a sequence of adjunctions and the Peter--Weyl presentation of $\mathcal{O}_q(\widetilde{G})$:
\begin{equation} \label{eq:qmm-explicit}
    \mu \in \mathrm{Hom}_{\A}(\mathcal{O}_q(\widetilde{G}), \iend_\M(\C)) \simeq \mathrm{Hom}_{\M}(\mathcal{O}_q(\widetilde{G}), \C) \simeq \bigoplus_{i \in I} \, \mathrm{Hom}_{\M}(V_i, V_i) \ni \mathcal{J},
\end{equation}
where $i \in I$ is an index set for irreducible representations of $\widetilde{G}$, and $\mathcal{J}$ is the universal $K$-matrix of the relevant quantum symmetric pair provided by Kolb's Theorem \ref{thm:Kolb-thm}. Indeed, the map $\mu$ may be represented diagrammatically as the left picture below, while the $\mathcal{J}$ matrix is presented in terms of the skein picture on the right:
\begin{equation}
    \begin{tikzpicture}[scale =1,baseline={([yshift=-.5ex]current bounding box.center)}]
    \cyl{1}{1.5}
    \draw (-.5,-.433) .. controls ($(-.5,-.433)+ (0,.2)$) and (1,0) .. (.6,.8) .. controls (0,1.5) and ($(-.8,-.291)+(0,.2)$) .. (-.8,-.291);
    \filldraw (-.5,-.433) circle (1.5pt);
    \filldraw (-.8,-.291) circle (1.5pt);
    \draw[white, line width = 0.2cm] (0,0) -- (0,1);
    \draw[redstrand] (0,0) -- (0,1.5);
\end{tikzpicture} 
\quad= \quad
\begin{tikzpicture}[scale =1,baseline={([yshift=-.5ex]current bounding box.center)}]
    \cyl{1}{1.5}
    \draw[red] (0,0) -- (0,0.8);
    \draw[red] (0,1) -- (0,1.8);
    \draw[draw=red] (-.2,0.8) rectangle ++(0.3,0.2) node [anchor=north west][inner sep=0.75pt]  [font=\tiny, color=red]  {$\mathcal{J}$};
    \draw (-.5,-.433) node [anchor=south west][inner sep=0.75pt]  [font=\tiny]  {$V$} .. controls (-.3,0) and (-.1,0.3) .. (-.1,0.8);
    \draw (-.8,-.291) node [anchor=south west][inner sep=0.75pt]  [font=\tiny]  {$V$} .. controls (-.8,0) and (-.3,2) .. (-.1,1);
\end{tikzpicture}
\end{equation}

\begin{rem}\label{rem:q=1-qmm}
    For specializable parameters, the involution induced by the adjoint action of $\mathcal{J}$ specializes to $\theta$ by \cite[Proposition 10.2]{Kol14}; in particular, $\mathcal{J}|_{q = 1} = \widetilde{J}$.
\end{rem}

Our goal is to show that the map $\mu$ is a finite map (and in fact surjective in the inner case). Towards this, we will first show that $\mu$ deforms the map $\Ocal(\widetilde{G})\to \Ocal(G^\theta \backslash G)$ induced by the finite map \eqref{eq:embedding}. To this end, we first set $q = 1$ and understand the classical limit. 

\begin{lem} \label{lem:qmm-q=1}
    The internal endomorphism ring $\iend(1_{\M})$ specializes at $q = 1$ to the ring of functions on the affine variety $G^\theta \backslash \widetilde{G}$, and the quantum moment map $\mu: \mathcal{O}_q(\widetilde{G}) \to \iend(\mathbf{1}_{\M})$ at $q = 1$ coincides with the pullback on functions along the finite map \eqref{eq:embedding}. 
\end{lem}
\begin{proof}
    At $q = 1$, the action functor $\mathrm{act}_{\C}$ is the restriction functor $\mathrm{Res}_{G^\theta}^{\widetilde{G}}$, while its right adjoint $\mathrm{act}_\C^R$ is the induction functor $\mathrm{Ind}_{G^\theta}^{\widetilde{G}}$. Thus, 
    $$\iend_{\M}(\C) = \mathrm{Ind}_{G^\theta}^{\widetilde{G}}(\C) \simeq \mathcal{O}(G^\theta \backslash \widetilde{G})$$
    is the target of the quantum moment map at $q = 1$.

    To identify $\mu$, we must carry the element $\mathcal{J}$ through the chain of identifications in \eqref{eq:qmm-explicit}, and we will go step by step from right to left. First, we view the $i$th component $\mathcal{J}_i \in \mathrm{Hom}_{G^\theta}(V_i, V_i)$ as an element of $\mathrm{Hom}_{G^\theta}(V_i^* \otimes V_i, \C)$ by the map
    \begin{equation} \label{eq:matrix-coeff-J}
        \phi \otimes v \longmapsto \phi(\widetilde{J}_iv),
    \end{equation}
    essentially extracting a matrix coefficient of the $i$th component of $\widetilde{J}$. Regarding now $\widetilde{J} = \oplus_i \widetilde{J}_i \in \mathrm{Hom}_{G^\theta}(\oplus_i \, V_i^* \otimes V_i, \C)$ as a $G^\theta$-invariant distribution on $\widetilde{G}$ (\ie an element of $\mathrm{Hom}_{G^\theta}(\mathcal{O}(\widetilde{G}), \C)$) via the Peter--Weyl embedding, we see that the element corresponding to $\widetilde{J}$ is simply $f \mapsto f(\widetilde{J})$, since it is so for matrix coefficients, which form a basis. Finally, we apply Frobenius reciprocity to obtain a corresponding element of $\mathrm{Hom}_G(\mathcal{O}(\widetilde{G}), \mathcal{O}(G^\theta \backslash \widetilde{G}))$, which is given by the formula $f \mapsto (g \mapsto f(g^{-1}\widetilde{J} g))$, as we wanted to show.
\end{proof}

We deduce readily that the finiteness at $q = 1$ is preserved for general $q$.
\begin{lem}\label{lem:qmm-surj}
The quantum moment map 
$$\mu:\Oq(\widetilde{G}) \longrightarrow \iend_{\M}(\C)$$
is finite, and in fact surjective when $\widetilde{G}^\theta = G^\theta$.
\end{lem}
\begin{proof}
   We shall assume for simplicity that we are in the inner case, as the outer case requires minimal modifications. Since $\mu$ is a morphism of $G$-representations, we may decompose the target into irreducible representations $\oplus_i \, W_i$, and for each $i$, we write $M_i$ for the cokernel of the composition
   $$\mathcal{O}_q(G) \overset{\mu}{\longrightarrow} \oplus_i \, W_i \overset{\mathrm{pr}_i}{\longrightarrow} W_i.$$
   Note that $M_i$ is a finitely generated $\C[q]_{(q-1)}$-module as it is a quotient of the finite rank module $W_i$. Since $M_i \otimes_{q \to 1} \C = 0$ by the preceding Lemma \ref{lem:qmm-q=1}, by Nakayama's lemma we may conclude that $M_i = 0$. 
\end{proof}
In particular, $\iend_{\M}(\C)$ is a finitely generated $\Oq(\widetilde{G})$-module and one computes readily that its GK-dimension is
\begin{equation}\label{eq:GKdim-G/K}
    \operatorname{GKdim}(\iend_{\M}(\C)) = \dim(G)-\dim(G^\theta)
\end{equation} 
at generic $q$.  

\begin{rem}
    The internal endomorphism ring $\iend_{\M}(\C)$ has appeared in Noumi's work on quantizations of symmetric spaces \cite{Noumi}; there, a $q$-deformation of $\mathcal{O}(G^\theta \backslash G)$ was obtained for quantum symmetric pairs of Type AI and AII.
\end{rem}

\subsection{Finiteness and holonomicity of defect skein modules}
\label{subsec:holonomic-def-skmod}

Recall that
$$\mathcal{L}_\M := \sk^{\mathrm{int}}_{(\A,\M)}(\mathbb{D}_* \times \mathbb{S}^1)$$
inherits a $\mathcal{D}_\A$-module structure by regarding 2-torus as the boundary of $\mathbb{D}_* \times \mathbb{S}^1$. Since $\mathcal{D}_\A \simeq \DqG$ in this case is the deformation quantization of the symplectic variety $G \times G$, it makes sense to consider holonomicity of modules over it. 

For completeness, let us recall briefly the theory of deformation quantization modules. Let $A$ be a smooth Poisson algebra over $k = \C(q^{1/d})$ and let $\mathbf{A}$ be an associative flat (\ie torsion free) deformation of $A$ over $R = \C[q^{\pm1/d}]_{(q-1)}$, so that the $(q-1)$-linear term of the commutator bracket on $\mathbf{A}$ agrees with the Poisson bracket on $A$. 

\begin{definition}\label{def:holonomic-brmod}
We write $\mathbf{A}_k = \mathbf{A} \otimes_{R} k$ and consider a module $\mathbf{N}_k$ over $\mathbf{A}_k$. We say that $\mathbf{N}_k$ is a \textit{holonomic module} if it is finitely generated over $\mathbf{A}_k$, and there exists a flat $R$-lattice $\mathbf{A} \acts \mathbf{N}$ such that the support of the coherent sheaf defined by $N = \mathbf{N} \otimes_{R} \C$ on $\spec A$ is a Lagrangian submanifold.
\end{definition}

We are now ready to state and prove one of our central results.
\begin{thm}\label{thm:QSP-holonomic}
    Let $\mathbf{c}, \mathbf{s}$ be specializable parameters for the braided module category $\mathcal{M} = \mathrm{Rep}_{\mathbf{c}, \mathbf{s}}G^\theta$ over $\mathcal{A} = \mathrm{Rep}_q(G)^{\Z_2}$. Then the defect skein module $\mathcal{L}_{\M}$ a holonomic $\Dq(\widetilde{G})$-module. 
\end{thm} 
\begin{proof}
We proceed by analyzing the limit when $q \to 1$. Recall the description of $\mathcal{L}_{\M}$ from Proposition \ref{prop:LM-formula}: in the commutative setting, we may compute
    \begin{equation*}
        \int^{x \in \mathrm{im}(F)}x \otimes \underline{\mathrm{End}}(1_\M) \otimes x^* \simeq \left(\int^{x \in \mathrm{im}(F)}x \otimes x^*\right) \otimes \underline{\mathrm{End}}(1_\M) \simeq \mathcal{O}(G^\theta \times G^\theta \backslash \widetilde{G}).
    \end{equation*}
As a $\mathcal{D}_{\A} \otimes_{q \to 1} \C \simeq \mathcal{O}(\widetilde{G} \times \widetilde{G})$-module, is induced by the finite map $G^\theta \times G^\theta\backslash \widetilde{G} \to G \times G$ having the inclusion in the first factor and the map \eqref{eq:embedding} in the second. In particular, $\mathcal{L}_{\M}$ is a finitely generated $\mathcal{D}_{\A}$-module at $q \to 1$ and by a similar application of Nakayama's Lemma as in Lemma \ref{lem:qmm-surj} we may conclude that $\mathcal{L}_{\M}$ is finitely generated over $\mathcal{D}_{\A}$. The GK-dimension of this $\mathcal{L}_{\M}$ is readily computed as 
$$\mathrm{dim}(G^\theta) + \mathrm{dim}(G^\theta \backslash G) = \mathrm{dim}(G),$$
which is what we wanted to show. 
\end{proof}

Besides the class of holonomic modules proposed by Definition \ref{def:holonomic-brmod}, which may be understood as having algebraic origin, an important class of holonomic $\mathcal{D}_\A$-modules of geometric origin was constructed by recent work of the second and third authors \cite[Theorem 1.2]{JR}, which we reproduce in the specific case most relevant to our current discussion. 
\begin{thm}[Jordan--Romaidis] \label{thm:Jordan-Romaidis}
    Let $G = \Gm, \mathrm{SL}_2$, or $\mathrm{GL}_2$. Suppose $q$ is generic, and let $M$ be a 3-manifold with boundary torus. Then the internal skein module $\sk_G^{\mathrm{int}}(M)$ is a holonomic module over $\DqG$. Moreover, if $\mathcal{L}$ is another holonomic $\DqG$-module, then \[\sk_G^{\mathrm{int}}(M)\otimes_{\DqG} \mathcal{L}\] is finite dimensional. 
\end{thm}
There is a more general version of the above Theorem for larger genus and derived tensor products, for which we refer to \cite{JR}. It is fully expected that the preceding theorem should hold for any reductive group $G$, with routine but tedious algebraic input. In the meantime, we will give a name for the conclusion of the holonomicity theorem of Jordan--Romaidis: given a reductive group $G$ and a $\langle\sigma\rangle$-extension $\widetilde{G} = G \rtimes \langle \sigma\rangle$, we say that \textit{the $\Dq(\widetilde{G})$-modules of geometric origin are holonomic} if the conclusion of Theorem \ref{thm:Jordan-Romaidis} holds. That is, for every 3-manifold $M$ with boundary a 2-torus, the tensor product $\sk_{\widetilde{G}}^{\mathrm{int}}(M) \otimes_{\Dq(\widetilde{G})} \mathcal{L}$ 
is finite dimensional for $\mathcal{L}$ a holonomic module over $\Dq(\widetilde{G})$. 
In fact, it is straightforward to show that if the $\DqG$-modules of geometric origin are holonomic, then so are the $\mathcal{D}_q(\widetilde{G})$-modules of geometric origin. 

\begin{cor} \label{cor:finiteness-defect-skmod}
    Let $(G,G^\theta)$ be an inner (resp. outer) symmetric pair where $G$ is reductive and let $\widetilde{G} = G\rtimes \langle \sigma \rangle$ denote the associated extension. Suppose that the $\Dq(\widetilde{G})$-modules of geometric origin are holonomic. Then for any closed oriented 3-manifold $M$ with line defect $K \subset M$, the defect skein module
    $$\sk_{(G,G^\theta)}(M,K)$$
    is finite dimensional. 
\end{cor}
\begin{proof}
    By Theorem~\ref{thm:int-sk-gluing} (see also Corollary \ref{cor:def-skmod}), we may express the defect skein module $\sk_{(G,G^\theta)}(M,K)$ as 
    \[\sk_{(G,G^\theta)}(M,K)\cong \Hom_{\langle\sigma\rangle}\left(\unit,  \sk^{\mathrm{int}}_{\widetilde{G}}(M\setminus K)\otimes_{\Dq(\widetilde{G})}\mathcal{L}_\M\right)\]
    where $\langle \sigma \rangle =1$ (resp.\ $\langle \sigma\rangle = \Z_2$) in the inner (resp.\ outer) case and $\widetilde{G} = G\rtimes \langle \sigma \rangle$. 
    Since $\mathcal{L}_\M$ is holonomic by Theorem \ref{thm:QSP-holonomic} and $\sk^{\mathrm{int}}_{\widetilde{G}}(M\setminus K)$ is holonomic by assumption (unconditionally in the inner case when $G = \Gm, \mathrm{SL}_2, \mathrm{GL}_2$ by Theorem \ref{thm:Jordan-Romaidis}), we conclude that the tensor product is holonomic over $k$ and hence finite dimensional. To conclude the proof, if $V$ is a finite dimensional $\langle \sigma \rangle$-representation, then its space of invariants $\Hom_{\langle \sigma \rangle}(\unit, V) = V^{\langle \sigma \rangle}$ is automatically finite dimensional. 
\end{proof}

\section{Equivariant skein theory and double affine Hecke algebras}

In certain situations, skein theory with quantum symmetric pair defects can be computed equivalently via \textit{$\Z_2$-equivariant factorization homology} on the level of surfaces \cite{Wee19}. This perspective is both of theoretical and computational importance, as it allows us to conceptualize certain defect skein calculations in terms of equivariant topology. In this section, we explicit the connection between $\Z_2$-equivariant skein theory and skein theory with codimension 2 defects, and we construct representations of the double affine Hecke algebra (DAHA) of Type $\mathrm{C}^\vee\mathrm{C}_n$ as an application of the skein theory with Type $\mathrm{AIII}$ defects. We remark that modules for the DAHA of Type $\mathrm{C}^\vee\mathrm{C}_n$ have been constructed by \cite{Jordan-Ma} using the same data via purely algebraic methods, and we leave to future work the full algebraic analysis of the (generalized) DAHA modules we construct here.

\subsection{$\Z_2$-equivariant skein theory}\label{subsec:equivariant-sk}

In this section, we will introduce and construct $\Z_2$-equivariant skein theory as a $\Z_2$-equivariant TQFT in a similar fashion to $\Z_2$-equivariant factorization homology \cite{Wee20}. In fact, we start by establishing a relation between stratified and equivariant factorization homology in certain situations and then construct $\Z_2$-equivariant skein modules. Finally, we apply these constructions to representation theory of DAHA of type $\mathrm{C}^\vee\mathrm{C}_n$.\\ 

Consider a smooth oriented $n$-manifold $X$ with an orientation preserving $\Z_2$-action (where $n = 2,3$ being the most important dimensions for skein theory). We will always assume that the action is generically stabilizer free, that the fix points of the action form a codimension 2 submanifold of $X$, and we write $X_{\Z_2}:=X/\Z_2$ for the orbifold quotient. 
Let $\operatorname{Man}^{\Z_2}_2$ denote the category of (oriented) surfaces with an (orientation-preserving) $\Z_2$-action and $\Z_2$-equivariant embeddings as morphisms. The subcategory $\operatorname{Disk}_2^{\Z_2}\subset \operatorname{Man}_2^{\Z_2}$ denotes the full subcategory where the surfaces are disjoint unions of disks. Finally, $\Bord_3^{\Z_2}$ denotes the category with objects in $\operatorname{Man}_2^{\Z_2}$ whose underlying surfaces are closed and morphisms given by bordisms with a $\Z_2$-action with $\Z_2$-equivariant boundary parametrizations. 

\subsubsection{Equivariant vs stratified factorization homology}
\label{subsec:equiv-skcat}
Following \cite{Wee19}, we consider two types of local pictures for $\Z_2$-equivariant factorization homology:
\begin{itemize}
    \item the disk $\mathbf{D}$, represented by a free action of $\Z_2$ swapping two identical copies of the 2-disk $\mathbb{D}$, and
    \item the marked disk $\mathbf{D}_*$, represented by the action of $\Z_2$ on the 2-disk $\mathbb{D}$ by rotation by $\pi$. 
\end{itemize}
Note that $\mathbf{D}$ may be regarded as a non-equivariant 2-disk by passing to the quotient by $\Z_2$, while $\mathbf{D}_*$ may be regarded as the non-equivariant 2-disk with codimension 2 defect $\mathbb{D}_*$, where the defect locus is precisely the $\Z_2$-fixed point at the center.

As shown in \cite{Wee20}, a $\Z_2$-equivariant disk-algebra in $\prl$ is exactly a triple $(\A,\Phi,\M)$ of: 
\begin{itemize}
    \item a ribbon tensor category $\A$ with a tensor involution $\Phi:\A\to \A$, \ie $\Phi^2\simeq \id$, corresponding the free $\Z_2$-action of $\mathbf{D}$ and 
    \item a $\Phi$-twisted braided module category $\M$ over $\A$, with $\Phi$-twisted braiding $e_{m,x}: m\ract x \xrightarrow{\sim} m\ract \Phi(x)$ subject to relations. 
\end{itemize}

\begin{rem}
   \begin{enumerate}
   \item 
   Note that $(\A,\id, \M )\equiv \AM$ is precisely the datum of a ribbon tensor category and a (untwisted) braided module category and thus coincides with our algebraic input for stratified factorization homology (resp.\ skein theory with codimension two defects). 
   \item For a general triple $(\A, \Phi, \M)$ one obtains the triple $(\A^\Phi, \id, \M)$ where $\A^\Phi$ is $\A$ if $\Phi$ is inner and the $\Z_2$-equivariantization of $\A$ by the $\Phi$-action if it is outer. The $\Phi$-twisted braiding of $\M$ induces canonically an (untwisted) braiding over $\A^{\Z_2}$. 
\end{enumerate}
\end{rem}

Following \cite{Wee20}, we may consider the following definition of $\Z_2$-equivariant factorization homology of $\Sigma$.  

\begin{definition} \label{def:equiv-fact-hom}
    The \textit{$\Z_2$-equivariant factorization homology} with coefficients in $(\mathcal{A}, \mathcal{M})$ of $\Z_2 \acts \Sigma$ is given by  the left Kan extension
    \[
    \begin{tikzcd}
        \operatorname{Disk}_2^{\Z_2}\arrow[rr,"{(\A,\Phi, \M)}"] \arrow[d, hook] & & \prl\\
        \operatorname{Man}_2^{\Z_2}\arrow[urr,swap,"{\int_{-}{(\A,\Phi,\M)}}"] &&
    \end{tikzcd}~.
    \]
\end{definition}

It is an immediate consequence of our definition that the ${\Z_2}$-equivariant factorization homology of $\Z_2 \acts \Sigma$ with coefficients in $(\mathcal{A}, \mathcal{M})$ may be canonically identified with (nonequivariant) defect factorization homology of $(\Sigma_{\Z_2}, P_{\Z_2})$ with coefficients in $(\mathcal{A}^{\Phi}, \mathcal{M})$, where $\Sigma_{\Z_2}$ is the (coarse) quotient surface and $P_{\Z_2} \subset \Sigma_{\Z_2}$ is the image of the (possibly empty) set of fixed points of ${\Z_2}$ on $\Sigma$:
\begin{equation}\label{eq:equiv-vs-def-fact-hom}
    \int_{\Z_2 \acts \Sigma} (\mathcal{A}^\Phi,\id, \mathcal{M}) \simeq \int_{(\Sigma_{\Z_2}, P_{\Z_2})} (\mathcal{A}^{\Phi}, \mathcal{M})
\end{equation}
Indeed, the (symmetric monoidal) quotient functor \[\operatorname{Man}_2^{\Z_2}\to \operatorname{Man}_2^{\mathrm{mkd}}, (\Z_2\acts \Sigma) \mapsto (\Sigma_{\Z_2}, P_{\Z_2})
\]
makes the following commute (up to natural isomorphism): 
\[
\begin{tikzcd}
        \operatorname{Disk}_2^{\Z_2}\arrow[rr,"{(\A,\Phi, \M)}"] \arrow[d, hook] & & \prl\\
        \operatorname{Man}_2^{\Z_2}\arrow[rr,"(-)_{\Z_2}"] && \operatorname{Man}_2^{\mathrm{mkd}} \arrow[u,swap,"{\int_{-}(\A^\Phi, \M)}"]
\end{tikzcd}~.
\]
Hence, it satisfies to show that the associated composite $\operatorname{Man}_2^{\Z_2}\to \prl$ satisfies equivariant factorization homology. This follows from the fact that ${\Z_2}$-equivariant embeddings from $\mathbf{D}$ and $\mathbf{D}_*$ to $\Sigma$ are equivalent to marked embeddings from $\mathbb{D}$ and $\mathbb{D}_*$ to $(\Sigma_{\Z_2}, P_{\Z_2})$.

\begin{rem}\label{rem:branched-covers}
    While the equivariant factorization homology of ${\Z_2} \acts \Sigma$ may be computed via the defect theory, and thus can admit skein-theoretic presentations as in Definition \ref{def:skcat}, the converse is not true. Indeed, given a surface with point defects, it is not always possible to find an equivalent ${\Z_2}$-equivariant situation, meaning a branched double cover whose branch points are exactly the point defects on the quotient. For instance, if $\Sigma$ is a closed genus $g$ surface with $n$ marked points, then Hurwitz's formula dictates that any connected branched $\Z_2$-cover $\widetilde{\Sigma} \to \Sigma$ of genus $\widetilde{g}$ must satisfy the topological condition
    $$2-2\widetilde{g}  = 4-4g-n \iff \widetilde{g} = 2g-1 +\frac{n}{2}.$$
    In other words, the number of marked points must be even, and we have a lower bound of $n \geq 2$ in the case when $g = 0$.
\end{rem}

\subsubsection{Equivariant skein theory}\label{subsec:equiv-skmod}

Having established the comparison between equivariant factorization homology and stratified factorization homology, we shift our focus to the construction of $\Z_2$-equivariant skein theory and thus the extension to $3$-manifolds. We restrict to the following type of $\Z_2$-actions on 3-manifolds. 
\begin{definition}\label{def:Galois-action}
    Let $M$ be a closed oriented 3-manifold with $\Z_2$-action. We say that the action is \textit{Galois} if the quotient map $M \to M_{\Z_2}$ is a branched covering. That is, the fixed points of $\Z_2$ form an embedded link (which we call the \textit{ramification link}) and the induced $\Z_2$-action on the unit normal bundle of the ramification link is given by $\pi$-rotation.   
\end{definition}

In the following, we will always assume that actions $\Z_2 \acts M$ are Galois. We write $\widetilde{K}_{\Z_2} \subset M$ for the ramification link, and $K_{\Z_2} \subset M_{\Z_2}$ for its (isomorphic) image link in the quotient $M_{\Z_2}$. 

We shall set up carefully the notion of $\Z_2$-equivariant skein modules of $\Z_2 \acts M$ with coefficients in $(\A,\Phi,  \M)$; we remind that we require $\M$ to be a $\Phi$-braided module category over the equivariantized category $\A$, and we shall denote the action of $\Z_2$ on $\A$ by $x \mapsto \Phi(x)$. 
\begin{definition}\label{eq:equiv-ribbons}
    An \textit{equivariant $(\A,\Phi,\M)$-ribbon graph} in a 3-manifold $M$ with Galois $\Z_2$-action is a finite $(\A,\Phi,\M)$-ribbon graph $\gamma$ in $M$ with tangle $\widetilde{K}_{\Z_2}$ (in the sense of Definition \ref{def:AM-ribbons} with the following condition: the action of $\Z_2$ leaves $\gamma$ invariant, and the action is equivariant on $\gamma \setminus \widetilde{K}_{\Z_2}$, \ie 
    \begin{equation} \label{eq:equiv-ribbon-graph}
        \sigma \cdot (\gamma \setminus \widetilde{K}_{\Z_2}) = \Phi(\gamma \setminus \widetilde{K}_{\Z_2}).
    \end{equation}
\end{definition}
Concretely, \eqref{eq:equiv-ribbon-graph} means the following: away from the ramification link, an equivariant ribbon graph is acted on simply transitively by $\Z_2$, and its $\A$-labels are flipped by $\Phi$ upon the applying the action of $\sigma$. 

\begin{definition}\label{def:equiv-skmod}
    Let $M$ be an oriented 3-manifold with a Galois $\Z_2$-action and ramification link $\widetilde{K}_\Gamma$. The \textit{equivariant $(\A,\M)$-skein module} is defined as
    $$\sk_{(\A,\Phi,\M)}(\Z_2 \acts M) := \frac{\langle \text{equivariant }(\A,\Phi,\M)\text{-ribbon graphs in }M \rangle}{\mathrm{lr}(_{\A^{\Phi}} \M), \mathrm{lr}(\A^{\Phi})}.$$
    \end{definition}
    Explicitly, the local relations are defined as kernels of the $\Z_2$-equivariant ribbon functor
    \begin{equation}
\mathrm{ev}_\A:\operatorname{Rib}_\A(\mathbf{D})\to \A
    \end{equation}
    resp.\ the $\Phi$-twisted braided module functor
    \begin{equation}
        \mathrm{ev}_\M: \operatorname{Rib}_\M(\mathbf{D}_\ast)\to \M~.
    \end{equation}
    Here, $\operatorname{Rib}_\A(\mathbf{D})$ is defined through $\Z_2$-equivariant ribbon graphs on the cylinder $\mathbf{D}\times [0,1]$ and its $\Z_2$-action is given by swapping the two disk cylinders of $\mathbf{D}\times [0,1]$.
    Similarly, $\operatorname{Rib}_\M(\mathbf{D}_\ast)$ is defined through $\Z_2$-equivariant ribbon graphs in $\mathbf{D}_\ast \times [0,1]$ and its $\Phi$-twisted braided structure over $\A$ is induced by $\pi$-rotation around the origin (see Figure~\ref{fig:z2-equiv-local-skein}).
    
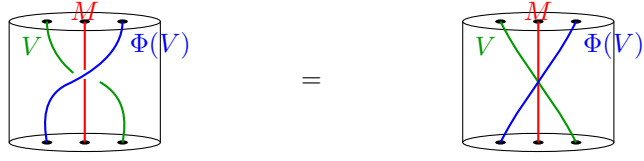
\begin{figure}
\begin{tikzpicture}[yscale=0.4]
\begin{scope}
\draw (-1,0) -- (-1,4);
\draw (1,0) -- (1,4);
\draw (0,4) ellipse (1 and 0.3);
\draw (0,0) ellipse (1 and 0.3);
\fill (-0.5,4) circle (2pt);
\fill (0,4) circle (2pt);
\fill (0.5,4) circle (2pt);
\fill (-0.5,0) circle (2pt);
\fill (0,0) circle (2pt);
\fill (0.5,0) circle (2pt);
\draw[red, thick] (0,4) .. controls (0,3.5) and (0,3) .. (0,2.4);
\draw[red, thick] (0,2.1) .. controls (0,1.8) and (0,1.5) .. (0,0);
\draw[green!60!black, thick]
(-0.5,4)
.. controls (-0.5,3) and (-0.2,2.4) ..
(-0.15,2.3);
\draw[green!60!black, thick]
(0.2,2)
.. controls (0.55,1.5) and (0.55,0.7) ..
(0.5,0);
\draw[blue, thick]
(0.5,4)
.. controls (0.5,3) and (0.1,2.3) ..
(-0.25,1.9)
.. controls (-0.55,1.5) and (-0.55,0.7) ..
(-0.5,0);
\node[green!60!black] at (-0.7,3.2) {$V$};
\node[blue] at (1,3.2) {$\Phi(V)$};
\node[red] at (0,4.35) {$M$};
\end{scope}
\node at (3,2) {$=$};
\begin{scope}[xshift=6cm]
\draw (-1,0) -- (-1,4);
\draw (1,0) -- (1,4);
\draw (0,4) ellipse (1 and 0.3);
\draw (0,0) ellipse (1 and 0.3);
\fill (-0.5,4) circle (2pt);
\fill (0,4) circle (2pt);
\fill (0.5,4) circle (2pt);
\fill (-0.5,0) circle (2pt);
\fill (0,0) circle (2pt);
\fill (0.5,0) circle (2pt);
\draw[red, thick] (0,4) -- (0,0);
\draw[green!60!black, thick]
(-0.5,4)
.. controls (-0.3,3) and (-0.1,2.4) ..
(0,2)
.. controls (0.1,1.6) and (0.3,1) ..
(0.5,0);
\draw[blue, thick]
(0.5,4)
.. controls (0.3,3) and (0.1,2.4) ..
(0,2)
.. controls (-0.1,1.6) and (-0.3,1) ..
(-0.5,0);
\node[green!60!black] at (-0.7,3.2) {$V$};
\node[blue] at (1,3.2) {$\Phi(V)$};
\node[red] at (0,4.35) {$M$};
\end{scope}
\end{tikzpicture}
\caption{Local relations replaces the $\Z_2$-equivariant skein on the left with the $\Z_2$-equivariant skein on the right with the coupon labelled by the braiding $e_{V,M}: V\lact M \xrightarrow{\sim}\Phi(V)\lact M$. }
\label{fig:z2-equiv-local-skein}
\end{figure}

\begin{prop}\label{prop:equiv-vs-def-skmod}
     Let $M$ be an oriented 3-manifold with a Galois $\Z_2$-action with ramification link $\widetilde{K}_{\Z_2}$. Writing $M_{\Z_2}$ for the (coarse) quotient of $M$ by $\Z_2$ and $K_{\Z_2} \subset M_{\Z_2}$ for the image of the ramification link. Then we have an identification of the $\Z_2$-equivariant skein module of $\Z_2 \acts M$ with the defect skein module of $(M_{\Z_2}, K_{\Z_2})$
    $$\sk_{(\mathcal{A}^\Phi,\id,\mathcal{M})}(\Z_2 \acts M) \simeq \sk_{(\mathcal{A}^{\Phi}, \mathcal{M})}(M_{\Z_2}, K_{\Z_2}).$$
\end{prop}
\begin{proof}
    The desired identification in fact lifts to a canonical isomorphism between the vector spaces generated by equivariant $(\A^\Phi,\id,\M)$-ribbon graphs in $\Z_2 \acts M$ and $(\A^\Phi,\M)$-ribbon graphs in the quotient $(M_{\Z_2}, K_{\Z_2})$. The local relations that one imposes in the bulk -- that is, away from the ramification link $\widetilde{K}_{\Z_2}$ in the equivariant case and away from the defect link $K_{\Z_2}$ in the defect case -- are also canonically identified by passing to the quotient by $\Z_2$. Finally, the local relations that one imposes in a neighbourhood of the ramification link (see Figure~\ref{fig:z2-equiv-local-skein}) can be identified with the local relation, see \eqref{eq:brmod-operad}, near the defect $K_{\Z_2}$. 
\end{proof}

\begin{rem}\label{rem:beta-fact-hom}
    In principle, instead of defining the (usual) skein module via generators and relations, it should be possible to integrate the ribbon category $\mathcal{A}$ over a closed 3-manifold via \textit{$\beta$-factorization homology} \cite{AFRbeta} to obtain a more categorical description of $\sk_\A$. With this technology, one would also be able to deduce the equivalence between equivariant and defect skein modules via categorical considerations as in the surface case \eqref{eq:equiv-vs-def-fact-hom}.
\end{rem}

\subsection{The DAHA of Type $\mathrm{C}^\vee\mathrm{C}_n$}
\label{subsec:DAHA}

Besides its evident utility in the topological theory of knots and links in 3-dimensions, quantum symmetric pairs have also featured prominently in the representation theory of double affine Hecke algebras (DAHAs) outside of type $\mathrm{A}_n$. We shall use the theory developed in \S \ref{subsec:equivariant-sk} to describe this application in topological terms, which leads immediately to generalizations that we investigate in detail in future work.

In \cite{JorDAHA} and \cite{Jordan-Ma}, quantizations of the DAHA representations of Calaque--Enriquez--Etingof \cite{CEE} in type $\mathrm{A}_n$ and Etingof--Freund--Ma \cite{EFM} in type $\mathrm{C}^\vee\mathrm{C}_n$ were produced, given the data of a quantum group and a quantum symmetric pair of Type AIII, respectively. The constructions in the Type $\mathrm{A}$ case can be conveniently phrased and conceptually generalized from a skein-theoretic perspective, as explained by \cite{GJVtori}. To recapitulate, we start by taking $\Sigma$ a closed oriented surface and letting $\mathrm{Br}_{\Sigma, n}$ denote the braid group on $n$ strands on $\Sigma$ (which may also be presented as the fundamental group of the configuration space of $n$ points on $\Sigma$). Choosing an object $V$ in a ribbon category $\A$, there is a homomorphism of algebras
\begin{equation}\label{eq:braid-group-action}
    \varphi_V: k[\mathrm{Br}_{\Sigma, n}] \longrightarrow \mathrm{End}_{\mathrm{Sk}_{\A}(\Sigma)}(\mathbf{V})
\end{equation}
where $\mathbf{V} \in \mathrm{Sk}_{\A}(\Sigma)$ denotes the object with $n$ points colored by $V$, and the action of an element of $\mathrm{Br}_{\Sigma, n}$ on $\mathbf{V}$ is obtained by stacking the braid on top of $\mathbf{V}$, regarded as an endomorphism of $\mathrm{Sk}_{\A}(\Sigma)$. When $\A = \mathrm{Rep}_q\mathrm{GL}_n$ and $\Sigma$ is the 2-disk, the annlus, and the 2-torus, the homomorphism $\varphi_V$ is known to factor through the finite, the affine, and the double affine Hecke algebras of Type $\mathrm{A}_n$, respectively, see Proposition 7.2 of \textit{op. cit}. In particular, we obtain a functor
\begin{equation} \label{eq:skcat-Hecke}
    \mathrm{Sk}_{\mathcal{A}}(\Sigma) \longrightarrow \mathrm{Rep}(\mathrm{Br}_{\Sigma, n})
\end{equation}
$$\mathbf{W} \longmapsto \mathrm{Hom}_{\mathrm{Sk}_{\A}(\Sigma)}(\mathbf{W}, \mathbf{V})$$
by pre-composition with $\varphi_V$. The image is not arbitrary: the image of \eqref{eq:skcat-Hecke} consists of $k[\mathrm{Br}_{\Sigma,n}]$-modules which factor through the skein relations.

We may generalize slightly the preceding construction to orbifold surfaces. To set things up, suppose $\Sigma$ is equipped with a Galois action by $\Z_2$, with orbifold quotient $\Sigma/\Z_2$, which we may regard equivalently as a surface $\Sigma_{\Z_2}$ with point defects. We will take $\A = \mathrm{Rep}_q(G)$ (or an equivariantization if needed) as input for the bulk skein theory on $\Sigma$, and we choose a quantum symmetric pair $\M = \mathrm{Rep}_{\mathbf{c},\mathbf{s}}(G^\theta)$ as input for the point defects. We obtain an analogue of \eqref{eq:skcat-Hecke} 
\begin{equation}\label{eq:skcat-Hecke-Z2}
    \mathrm{Sk}_{\A, \M}(\Sigma_{\Z_2}) \longrightarrow \mathrm{Rep}(\mathrm{Br}_{\Sigma_{\Z_2},n})
\end{equation}
where $\mathrm{Br}_{\Sigma_{\Z_2},n}$ denotes the braid group on $n$ strands for the orbifold surface, or equivalently the fundamental group of the configuration of $n$ nonsingular points on the orbifold $\Sigma_{\Z_2}$ (see Definition \cite{Wee19} for a more expanded definition). The image of this functor consists of modules of the group ring $k[\mathrm{Br}_{\Sigma_{\Z_2},n}]$ modulo local skein relations imposed by the quantum symmetric pair $(\A,\M)$, and we may regard this quotient as an analogue of the DAHA associated to a $\Z_2$-equivariant surface and a choice of symmetric pair $(G,\theta)$. 

The most interesting case, which directly relates to the DAHA of Type $\mathrm{C}^\vee\mathrm{C}_n$ can be constructed as follows. Recall that the DAHA of Type $\mathrm{C}^\vee\mathrm{C}_n$ with six parameters $(v,t,t_0,t_n,u_0,u_n)$, which we denote by $\mathbb{H}_{\mathrm{C}^\vee\mathrm{C}_n}(v,t,t_0,t_n,u_0,u_n)$, can be presented as a quotient of the group ring (over a $k$-algebra $\kk$ containing the indeterminants $(v,t,t_0,t_n,u_0,u_n)$) of the \textit{double affine braid group of Type } $\mathrm{C}^\vee\mathrm{C}_n$, denoted $\mathrm{Br}_{\mathrm{C}^\vee\mathrm{C}_n}$. The latter can be presented with generators $T_i$ for $i = 0, \ldots, n$ and $X_j,Y_j$ for $j = 1, \ldots, n$, modulo the following relations:
\begin{itemize}
    \item (Braid relations of Type $\mathrm{C}^\vee\mathrm{C}_n$). \begin{equation}
 \begin{matrix}T_iT_j = T_jT_i \text{ if } |i-j| > 1,\\[1mm] T_iT_{i+1}T_i = T_{i+1}T_iT_{i+1} \text{ for } i = 1,\ldots, n-2, \\[1mm] T_0T_1T_0T_1 = T_0T_1T_0T_1, T_{n-1}T_nT_{n-1}T_n = T_n T_{n-1}T_n T_{n-1} \end{matrix}
\end{equation}
    \item (Fundamental group of the torus). 
     \begin{equation}   X_iX_j = X_jX_i, \, Y_iY_j = Y_j Y_i \text{ for } i,j = 1,\ldots, n,
    \end{equation}
    \item (Elliptic braid relations).
    \begin{equation}
        \begin{matrix}T_i X_j = X_j T_i, \, T_iY_j = Y_jT_i \, \text{ if } |i-j| > 1 \text{ or if } (i,j) = (n,n-1), \\[1mm] T_iX_iT_i = X_{i+1}, \quad T_iY_{i+1}T_i = Y_i \, \text{ for } i = 1, \ldots, n-1.\\[1mm] X_j(P^{-1}Y_1) = (P^{-1}Y)X_j \text{ for } j = 2,\ldots, n-1.\end{matrix} 
    \end{equation}
    where $P := T_1\cdots T_{n-1}T_nT_{n-1}\cdots T_1$.
\end{itemize}
Alternatively, we may present $\mathrm{Br}_{\mathrm{C}^\vee\mathrm{C}_n}$ topologically as the (orbifold) fundamental group of the configuration space of $n$ unordered \textit{nonsingular} points on the $\Z_2$-orbifold surface obtained from the hyperelliptic involution on a 2-torus
$$\Sigma = \R^2/\Z^2, \text{ with involution } \sigma: (x,y) \mapsto (-x,-y).$$
The action of $\Z_2$ has 4 singular points (corresponding to the 2-torsion points on $\Sigma$), with $T_0, T_n$ corresponding to loops around two of them, while the $X_j, Y_j$'s correspond to the meridional and longitudinal loops around the $j$th copy of $\Sigma$. Note then that the loops around the remaining two singular points may be expressed in terms of the existing symbols as
$$X_n^{-1}T_n^{-1} \text{ and } Y_1^{-1}P X_1.$$

The DAHA $\mathbb{H}_{\mathrm{C}^\vee\mathrm{C}_n}(v,t,t_0,t_n,u_0,u_n)$ is then defined as the quotient of $\kk[\mathrm{Br}_{\mathrm{C}^\vee\mathrm{C}_n}]$ modulo the following Hecke relations \cite[Corollary 10.9]{Jordan-Ma}
\begin{equation} \label{equation Hecke relations}
    \begin{matrix}
         T_i \sim t \,  \text{ for } i = 1,\ldots, n-1, \\[1mm] T_0 \sim t_0, \quad T_n \sim t_n, \quad X_n^{-1}T_n^{-1} \sim u_n, \quad v^{-1}Y_1^{-1}P X_1 \sim u_0,
    \end{matrix} 
\end{equation}
where we have written $X \sim x$ to mean the quadratic Hecke relation $X^2 - (x-x^{-1})X +1 = 0$, as is customary.

\begin{prop} \label{prop: DAHA}
    Let $\Sigma$ be the 2-torus with $\Z_2$ acting by hyperelliptic involution, and let $(\A,\M)$ be a quantum symmetric pair of Type AIII. Let $\mathbf{V}_{\Z_2}\in \sk_\AM(\Sigma_{\Z_2})$ be the object with $n$ equivariant points in the regular part of $\Sigma$ and $\unit\in\M$ labelling the singular points.  Then \eqref{eq:skcat-Hecke-Z2} induces a representation of $\mathrm{Br}_{\mathrm{C}^\vee\mathrm{C}_n}$ which factors through the DAHA $\mathbb{H}_{\mathrm{C}^\vee\mathrm{C}_n}$:
    \[\begin{tikzcd}
        \kk[\mathrm{Br}_{\mathrm{C}^\vee\mathrm{C}_n}] \arrow[r]\arrow[d]& \End(V_{\Z_2}) \\
        \mathbb{H}_{\mathrm{C}^\vee\mathrm{C}_n}(v,q,s_0,s_1,s_2,s_3)\arrow[ur]& 
    \end{tikzcd}\]
    where $s_0, s_1,s_2,s_3 \in k$ are the chosen parameters for the Type $\mathrm{AIII}$ quantum symmetric pair at the defect points corresponding to $T_0, T_n, X_n^{-1}T_n, Y_1^{-1}PX_1$, respectively.
\end{prop}
\begin{proof}
    It suffices to verify that the Hecke relations of \eqref{equation Hecke relations} are included in the local skein relations imposed by the quantum symmetric pair. Note that the first set of relations correspond to the usual quadratic Hecke relations for simple braids, imposed by the $R$-matrix of the ribbon category on the fundamental representation of Type $\mathrm{A}_n$. The three remaining relations, involving the parameters $t_0, t_n, u_0, u_n$, correspond to local skein relations imposed by the choice of defect at each of the 4 singular points. These parameters specialize to $s_0,s_1,s_2,s_3$ respectively according to the defect skein relation \eqref{eq:N=p+q-skein-relations}.
\end{proof}

\begin{rem}
    By taking $\Sigma = \mathbf{D}$ to be the 2-disk with $\Z_2$ acting by $\pi$-rotation, one can analogously construct modules over the Hecke algebra of Type $\mathrm{C}^\vee\mathrm{C}_n$. 
\end{rem}

\section{Perspectives from shifted quantization}\label{sec:shifted-quantization}

In this section, we situate the construction of our skein theory with line defects in the context of \textit{shifted geometric quantization} \cite{Safronov1}. While it is logically independent to skein-theoretic applications, we find it nonetheless conceptually clarifying. Furthermore, understanding the semi-classical picture, or $q \to 1$ limit, in terms of shifted symplectic geometry will be useful for future applications (and for some concrete computations, see Section \ref{sec:examples}). We will use freely the notions discussed in \textit{loc.\ cit}, with the relevant shifts being $n = -1,0,1,2$. 

\subsection{3-manifolds}\label{subsec:3-mfds}
The fundamental symplectic geometry picture was already leveraged in the proof of finiteness for skein modules without defects \cite{GJS}. We consider a closed oriented 3-manifold $M$ and a defect knot $K \subset M$, whose tubular neighborhood we denote by $\nu(K)$. We make an identification of $\partial\, \nu(K) \simeq T^2$ with the 2-torus, and we write $\mathring{M} := M \setminus \nu(K)$ for the knot complement. By restricting local systems on $\mathring{M}$ and $\nu(K)$ to their boundary $T^2$, we obtain the fundamental diagram
\begin{equation}\label{eq:intersection-Langrangian-diagram}
\begin{tikzcd}
	{\mathrm{Loc}_G(\mathring{M})} && {\mathrm{Loc}_G(\nu(K))} \\
	& {\mathrm{Loc}_G(T^2)}
	\arrow["a"', from=1-1, to=2-2]
	\arrow["b", from=1-3, to=2-2]
\end{tikzcd}
\end{equation}
The target of the maps $a$ and $b$, the moduli of Betti $G$-local systems on $T^2$, can be regarded as a double loop space of the 2-shifted symplectic stack $\mathrm{B}G$. As such, it acquires a 0-shifted symplectic structure by transgression (the Atiyah--Bott--Goldman symplectic structure), and it is well-known that the morphisms $a$ and $b$ are naturally equipped with 0-shifted Lagrangian structures. The Lagrangian intersection of $a$ and $b$ gives the moduli of Betti $G$-local systems on $M$ itself 
\begin{equation} \label{eq:inter-Langrangian}
    \mathrm{Loc}_G(M) \simeq \mathrm{Loc}_G(\mathring{M}) \times_{\mathrm{Loc}_G(T^2)} \mathrm{Loc}_G(\nu(K))
\end{equation}
which is then equipped with a $(-1)$-shifted symplectic structure (the Batalin–Vilkovisky symplectic structure). By the principle of $(-1)$-shifted deformation quantization, one attaches a vector space invariant to $\mathrm{Loc}_G(M)$, which in skein theory is interpreted as the skein module $\sk_G(M)$. 

The starting point for the present project can then be phrased as the following question: how can we replace the $0$-shifted Lagrangian morphism $b$ with another in order to obtain vector space invariants of the knot $K \subset M$? Given such a morphism $b': L \to \mathrm{Loc}_G(T^2)$, we may then form the \textit{skein module with line defect}
$$\text{``}\sk_{G,L}(M, K)\text{"} := \text{geom. quantization of }\mathrm{Loc}_G(\mathring{M}) \times_{\mathrm{Loc}_G(T^2)}  L.$$
By AKSZ transgression, there are three simple options to try:
\begin{enumerate}
    \item (Transgressing twice). Start with a 2-shifted Lagrangian $L' \to \mathrm{B} G$. Double looping, we can replace $b$ by the $0$-shifted Lagrangian 
    $$b': L =\mathrm{Map}(T^2, L') \to \mathrm{Loc}_G(T^2).$$ The only examples of such Lagrangians known to us arise from parabolic subgroups: if $P \subset G$ is a parabolic subgroup with Levi subgroup $H$, then $\mathrm{B}P \to \mathrm{B}H \times \mathrm{B}G$ carries the structure of a 2-shifted Lagrangian. The \textit{skein modules with parabolic defects} proposed and studied by \cite{BrownJ}\footnote{More precisely, in \cite{BrownJ} what was studied was skein theory with codimension 1 defects, or interfaces. In the case when the interface is a $T^2$ and bounds a solid torus, one may regard it as a line defect from the perspective of the bulk.} are skein modules with line defect obtained by a quantization of Lagrangians intersections of this form. 
    \item (Transgressing once). Start with a 1-shifted Lagrangian $L' \to \mathrm{Map}(\mathbb{S}^1,\mathrm{B}G) \simeq G/G$; by \cite[\S 2.3]{Safronov1}, this is equivalent to giving a \textit{quasi-Hamiltonian $G$-action}. Looping once, we can replace $b$ by the $0$-shifted Lagrangian 
    $$b': L = \mathrm{Map}(\mathbb{S}^1,L') \to \mathrm{Map}(\mathbb{S}^1, G/G) \simeq \mathrm{Loc}_G(T^2)~.$$
    The simplest example of such is when $L' = \mathrm{B}G \to G/G$ by the inclusion at the identity (or more generally, the inclusion of a conjugacy class). Then the morphism $b'$ recovers the Lagrangian $b$ of \eqref{eq:intersection-Langrangian-diagram}, and we obtain skein theory with the transparent defect.
    \item (Transgressing zero times). This is the case studied in the present article: one directly looks for a 0-shifted Lagrangian inside $b: L \to \mathrm{Loc}_G(T^2)$ without transgression. Those that are relevant for our skein modules with QSP defects are obtained, essentially, by a derived analogue of an observation of Beauville \cite{Beauville}: the fixed points of an antisymplectic involution on a symplectic manifold form a Lagrangian submanifold. 

    In forthcoming work of the first author with Š. Kaubrys \cite{EricSarunas}, we explain and provide examples of Beauville's observation in the derived setting, among which we may consider the following. Consider a group involution $\theta$ of $G$, and the \textit{orientation reversing} automorphism of $T^2$ given by the matrix $\gamma = \begin{bmatrix} 1 & 0\\0 & -1\end{bmatrix}$. By precomposing with $\gamma$ and postcomposing with $\theta$, we obtain an \textit{antisymplectic involution} $\Theta$ on $\mathrm{Loc}_G(T^2)$, whose fixed point stack is a 0-shifted Lagrangian 
    $$b': L =  \mathrm{Loc}_G(T^2)^\Theta \to \mathrm{Loc}_G(T^2).$$ The skein modules with QSP line defects studied in the present article are obtained by a quantization of Lagrangian intersections of this form.
\end{enumerate}

Let us focus now on Case (3) since it is of most relevance for us, and describe concretely the $(-1)$-shifted symplectic stack our skein module with defect quantizes. The analogue of \eqref{eq:inter-Langrangian} is now 
\begin{equation}\label{eq:intersection-our-Langrangian}
    \mathrm{Loc}_{G,\theta}(M,K) := \mathrm{Loc}_G(\mathring{M})\times_{\mathrm{Loc}_G(T^2)} \mathrm{Loc}_G(T^2)^\Theta,
\end{equation}
whose underlying classical stack parametrizes $G$-local systems $\rho$ on $\mathring{M}$ whose restriction to the boundary $T^2$ is equipped with an isomorphism
\begin{equation} \label{eq:moduli-description}
    \iota: \rho \circ \gamma \simeq \theta \circ \rho.
\end{equation}
Note that in particular, if we take the loop $\ell \in \pi_1(\mathring{M})$ which traverses $T^2$ along a meridian (\ie in the direction where $\gamma$ has eigenvalue 1), the previous equation implies that $\rho(\ell)$ lies in (a conjugate of) the symmetric subgroup $G^\theta$. Similarly, the loop $\ell' \in \pi_1(\mathring{M})$ which traverses $T^2$ along the longitude (\ie in the direction where $\gamma$ has eigenvalue $-1$), the previous equation implies that $\theta(\rho(\ell'))$ is conjugate to $\rho(\ell')^{-1}$ in $G$.

\begin{rem}\label{rem:equiv-moduli-space}
    This description is compatible with the equivariant setting developed in \S \ref{subsec:equivariant-sk}: suppose the 3-manifold with defect $(M,K)$ was obtained from a $\Z_2$-Galois action on some closed oriented 3-manifold $N$ with ramification link $K$. Near the preimage of $\nu(K)$ in $N$, the Galois action is exactly described by the matrix $\gamma$. Consequently the moduli stack $\mathrm{Loc}_{G,\theta}(M)$ can be equivalently described as the moduli stack of $\Z_2$-equivariant maps from $N$ to $\mathrm{B}G$, intertwining the action of $\theta$ on the target:
$$\mathrm{Loc}_{G,\theta}(M,K) \simeq \mathrm{Loc}_G(N)^{\Z_2}.$$
Since $\mathrm{Loc}_G(N)$ is $(-1)$-shifted symplectic and $\Z_2$ acts symplectically, by a ``$+$-version" of Beauville's observation \cite{EricSarunas}, we may deduce that the above moduli stacks are equipped with $(-1)$-shifted symplectic structures as well. 
\end{rem}

The algebraic ring of functions on $\mathrm{Loc}_{G,\theta}(M)$ thus gives a $q \to 1$ degeneration for the defect skein module $\sk_{G,\theta}(M,K)$, just as the algebraic ring of functions on $\mathrm{Loc}_G(G)$ gives a $q \to 1$ degeneration for the skein module without defect, as follows. Given a $(G,\theta)$-labeled ribbon graph $\Gamma$ in $(M,K)$, we regard its connected components as elements of the fundamental group of the knot complement $a_1, \cdots \, a_k, b_1, \ldots b_s \in \pi_1(\mathring{M})$ (where the $a$'s are bulk components and the $b$'s are defect components) colored by objects $V_1, \ldots, V_k$ of $\mathrm{Rep}_q(G)$ and $W_1, \ldots, W_s$ of $\mathrm{Rep}_{\mathbf{c,s}}(G^\theta)$ for the bulk and defect components respectively (we push off its defect component using the framing so that it lies on $\partial \mathring{M}$). We may thus construct an algebraic function on
$\mathrm{Loc}_{G,\theta}(M,K)$ by
\begin{equation}f_\Gamma(\rho, \iota) := \mathrm{tr}_\rho(\Gamma) = \sum_{i=1}^k \, \mathrm{tr}_{V_i}(\rho(a_i)) + \sum_{j=1}^s \mathrm{tr}_{W_j}(\rho(b_j)),
\end{equation}
where the trace of the $\rho(b_j)$'s on $W_j$ make sense since $\rho(b_j) \in G^\theta$ up to conjugacy by the discussion following \eqref{eq:moduli-description}.

\subsection{2-manifolds} The picture for surfaces is completely analogous, and we spell it out for clarity. We consider a closed oriented surface $\Sigma$ and defect points $x = \{x_1, \ldots, x_k\} \subset \Sigma$, whose tubular neighborhood $\nu(x)$ is a disjoint union of $k$ 2-disks. We write $\mathring{\Sigma} = \Sigma \setminus \nu(x)$, whose boundary is given by $k$ disjoint circles. The analogue of the fundamental diagram \eqref{eq:intersection-Langrangian-diagram} is
\begin{equation}
    \begin{tikzcd}
	{\mathrm{Loc}_G(\mathring{\Sigma})} && {\mathrm{Loc}_G(\nu(x))} \\
	& {\mathrm{Loc}_G(\partial \mathring{\Sigma}) \simeq (G/G)^{\sqcup k}}
	\arrow["a"', from=1-1, to=2-2]
	\arrow["b", from=1-3, to=2-2]
\end{tikzcd}
\end{equation}
In this case, the numerology of shifted symplectic geometry goes up by 1: the target of the maps $a$ and $b$ can be regarded as a loop space of the 2-shifted symplectic stack $\mathrm{B}G$, and as such it acquires a 1-shifted symplectic structure. It is well-known that the morphisms $a$ and $b$ are naturally equipped with 1-shifted Lagrangian structures (via the Alekseev--Malkin--Meinrenken quasi-Hamiltonian moment map, see \S 9 of \cite{AMM97}), and consequently the Lagrangian intersection of $a$ and $b$ gives the moduli of Betti $G$-local systems on $\Sigma$ itself
\begin{equation} \label{eq:intersection-Langrangian-2}
    \mathrm{Loc}_G(\Sigma) \simeq \mathrm{Loc}_G(\mathring{\Sigma}) \times_{\mathrm{Loc}_G(\partial \mathring{\Sigma})} \mathrm{Loc}_G(\nu(x)),
\end{equation}
along with its 0-shifted (Atiyah--Bott--Goldman) symplectic structure. By the principle of 0-shifted deformation quantization, one attaches a (1-)categorical invariant to $\mathrm{Loc}_G(\Sigma)$, which in skein theory is interpreted as the skein category $\sk_G(\Sigma)$.

Again, one may pose the following question: how can we replace the 1-shifted Lagrangian morphism $b$ to obtain (1-)categorical invariants of the defect surface $(\Sigma, x)$? Given such a Lagrangian $b': L \to \mathrm{Loc}_G(\partial \mathring{\Sigma})$, we may then form the \textit{skein category with point defect}
$$\text{``}\sk_{G,L}(\Sigma, x)\text{"} := \text{geom. quantization of }\mathrm{Loc}_G(\mathring{\Sigma}) \times_{\mathrm{Loc}_G(\partial \mathring{\Sigma})}  L.$$

By AKSZ  transgression, there are two simple options to try:
\begin{enumerate}
    \item (Transgressing once). Start with a 2-shifted Lagrangian $L' \to \mathrm{B}G$, we may loop once to obtain a 1-shifted Lagrangian 
    $$b': L = \mathrm{Map}(\mathbb{S}^1, L') \to \mathrm{Map}(\mathbb{S}^1, \mathrm{B}G) = G/G.$$
    Using again the 2-shifted Lagrangians afforded by parabolic subgroups, the aforementioned work of Brown--Jordan \cite{BrownJ} quantizes the resulting Lagrangian intersection into \textit{skein categories with parabolic defects}. 
    \item (Transgressing zero times). This is the case studied in the present article: one directly looks for a 1-shifted Lagrangian in $G/G$. As mentioned in the previous section, such Lagrangians can be produced form quasi-Hamiltonian moment maps, but for the QSP defect we apply once more the derived analogue of Beauville's observation \cite{EricSarunas}: with the group involution $\theta$ of $G$ and the \textit{orientation reversing} automorphism given by inversion on the circle, we obtain an \textit{antisymplectic involution} $\Theta$ on $\mathrm{Loc}_G(\partial \mathring{\Sigma})$ whose fixed point stack is a 1-shifted Lagrangian
    $$b': L = \mathrm{Loc}_G(\partial \mathring{\Sigma})^\Theta \longrightarrow \mathrm{Loc}_G(\partial \mathring{\Sigma}).$$
    The skein categories with QSP defects studied in the present article are obtained by a quantization of Lagrangian intersections of this form.
\end{enumerate}
Focusing on Case (2), the analogue of \eqref{eq:intersection-our-Langrangian} is now
\begin{equation}\label{equation intersection of our Lagrangians 2}
    \mathrm{Loc}_{G,\theta}(\Sigma, x):=  \mathrm{Loc}_G(\mathring{\Sigma}) \times_{\mathrm{Loc}_G(\partial\mathring{\Sigma})} \mathrm{Loc}_G(\partial \mathring{\Sigma})^\Theta,
\end{equation}
whose underlying classical stack parametrizes $G$-local systems $\rho$ on $\mathring{\Sigma}$ whose restriction to the boundary is equipped with an isomorphism
\begin{equation}
    \iota: \rho \circ \mathrm{inv} \simeq \theta \circ \rho.
\end{equation}
Note that in particular, if we take a monodromy loops $\ell$ around a boundary component of $\mathring{\Sigma}$, then the previous equation implies that $\theta(\rho(\ell))$ is conjugate to $\rho(\ell)^{-1}$ in $G$.

\begin{rem}
    The description is compatible with the equivariant setting developed in \S\ref{subsec:equivariant-sk}: suppose the surface with defect $(\Sigma, x)$ was obtained from a $\Z_2$-branched cover on some closed oriented surface $\widetilde{\Sigma}$, branching exactly at $x$. The moduli stack $\mathrm{Loc}_{G,\theta}(\Sigma, x)$ can be equivalently described by the moduli stack of $\Z_2$-equivariant maps from $\widetilde{\Sigma}$ to $\mathrm{B}G$, intertwining the action of $\theta$ on the target:
    $$\mathrm{Loc}_{G,\theta}(\Sigma, x) \simeq \mathrm{Loc}_G(\widetilde{\Sigma})^{\Z_2}.$$
\end{rem}

The $q \to 1$ degeneration for skein categories can be directly understood in terms of factorization homology:
$$\int_{\Sigma} \, \mathrm{Rep}_q(G) \overset{q \to 1}{\longrightarrow} \int_{\Sigma} \mathrm{Rep}(G) \simeq \mathrm{QC}(\mathrm{Loc}_G(\Sigma)).$$
We can proceed analogously for defects yielding
$$\int_{(\Sigma, x)} (\mathrm{Rep}_q(G), \mathrm{Rep}_{\mathbf{c,s}}(G^\theta)) \overset{q \to 1}{\longrightarrow} \int_{(\Sigma, x)} (\mathrm{Rep}(G), \mathrm{Rep}(G^\theta)) \simeq \mathrm{QC}(\mathrm{Loc}_{G,\theta}(\Sigma, x)).$$

\section{Examples} \label{sec:examples}

In this final section, we perform some experimental computations using defect skein theory from quantum symmetric pairs, and relate them to more classical constructions whenever possible.

\subsection{Electric 1-form symmetry}

We analyze first a degenerate case, which was first introduced in \cite[Example 3.6]{JLanglands}. We consider a central element $J \in G$ and take the quantum symmetric pair associated to the involution $\theta = \mathrm{Ad}_J$. Evidently, the fixed subgroup $G^\theta = G$ is the whole group, but the associated braided module $\mathcal{M} = \mathrm{Rep}_q(G)$ may be nontrivial: if $\sigma$ is the braiding on the braided category $\mathcal{A} = \mathrm{Rep}_q(G)$, then the braided module structure is given by
\begin{equation}
    e_{a \otimes m} = (\mathrm{id}_a \otimes J) \circ \sigma_{m \otimes a} \circ \sigma_{a \otimes m}
\end{equation}
for $a \in \A$ and $m \in \M$. The $q \to 1$ limit of the quantum moment map $\mu: \mathcal{O}_q(G) \to \mathcal{O}_q(G^\theta \backslash G) \simeq \C[q]_{(q-1)}$ is induced by the inclusion of the inclusion of the point $\{J\} \hookrightarrow G$. Given a defect knot $K$ inside a closed 3-manifold $M$, the associated defect skein module is exactly the twisted skein modules of Definition 3.8 in \textit{loc. cit}. In particular, assuming that $\DqG$-modules of geometric origin are holonomic, we may conclude from Corollary \ref{cor:finiteness-defect-skmod} that defect $G$-skein modules are finite dimensional.

\subsection{Dichromatic link invariants}\label{subsec:dichromatic}

Let us focus consider the case $G = \mathrm{SL}_2$ with the quantum symmetric pair defined by
$$J^t = q^{-1}\cdot \begin{bmatrix} t+t^{-1} & \sqrt{-1}\\ \sqrt{-1} & 0\end{bmatrix}$$
where $t \in R = \C[q^{\pm 1/2}]_{(q-1)}$ is an element specializing to (a fixed choice of) $\sqrt{-1}$ as $q \to 1$. The associated symmetric subgroup is conjugate to the diagonal torus, the fixed subgroup of $\theta = \mathrm{Ad}_J$. In the bulk, we impose the usual Kauffman bracket skein relations (see Definition \ref{definition KB skein relations}) along with the defect skein relations
\begin{gather} \label{eq:SL2-defect-relation-again}
 \brmodstrandsunoriented\; + q^{-2} \;  \invbrmodstrandsunoriented \;=\;q^{-1}(t+t^{-1})\;\parallelstrands{strand}{redstrand} \, , \quad \quad \brmodtraceunoriented = (t+t^{-1})  ~  \, \singlestrand{redstrand} 
\end{gather}

\begin{example}\label{eg:unknot}
    We may consider the simplest case of a defect unknot $K$ in the 3-sphere $M = \mathbb{S}^3$. Then the Kauffman bracket skein relations and the local defect relations \eqref{eq:SL2-defect-relation-again} are sufficient for the evaluation of a 2-variable knot polynomial in the variables $q,t$ in $\mathbb{S}^3 \setminus K$, \ie links in the solid 3-torus $\mathbb{S}^1 \times \mathbb{D}^2$. In particular, 
    $$\mathrm{dim}_k \, \sk(\mathbb{S}^3,K) = 1.$$
    Such \textit{dichromatic knot invariants} were previously studied by Hoste--Kidwell \cite{Hoste-Kidwell}, Lambropoulu \cite{Lambropoulou} and Hoste--Przytycki \cite{Hoste-Przytycki}. We point out, as was already observed in \cite[Corollary 3.8]{Hoste-Kidwell}, that if we only impose the Kauffman bracket skein relations in the bulk one does \textit{not} obtain a finite dimensional skein module. 
\end{example}

\subsection{Type AIII}
\label{subsec:AIII}
More generally, we consider the group $G = \mathrm{GL}_n$ with the symmetric subgroup $G^\theta = \mathrm{GL}_a \times \mathrm{GL}_b$ with $a+b = n$. 
The $J$-matrix constructed in \cite{Jordan-Ma} of the associated quantum symmetric pair depends on a specializable parameter $s \in R$ and is given by 
\begin{equation}\label{eq:Js}
J^s = \sum_{1\leq k\leq a}(s-s^{-1})E_k^k - \sum_{a+1\leq k\leq b} s  E^k_k + \sum_{1\leq k\leq a} (E_k^{n-k+1} + E_{n-k+1}^k)\end{equation}
Note that at $s = 1$, the matrix $J^s$ is conjugate to $\mathrm{diag}(\mathrm{Id}_a, -\mathrm{Id}_b)$, so the fixed subgroup is indeed conjugate to the block-diagonally embedded $\mathrm{GL}_a \times \mathrm{GL}_b \subset G$. A similar construction can be given for $\mathrm{SL}_n$ when $b$ is even. The associated Satake diagram of this symmetric subgroup is $X = \{a+1, \ldots, b\}$ with $\tau = \tau_0$ and thus it is an example of an inner symmetric pair. 

We present a diagrammatic approach to skein modules with such AIII type defects for $G=\GL_n$ by extending the $\GL_n$-diagramatics which involve $\GL_n$-webs as follows. 

\begin{definition}\label{def:web-skein}
    Let $M$ be an oriented 3-manifold with an embedded link $K \subset M$. A \textit{basic $(\GL_n, \GL_a\times \GL_b)$-skein} is an embedded oriented ribbon graph in $M\setminus K$ with two types of edges (solid and dotted). Vertices are given of the following three types:
\begin{equation}
    \begin{tikzpicture}[scale = 1.5, baseline={([yshift=-.5ex]current bounding box.center)}]
    \draw[->-=.5] (0,0) -- (.5,.5);
    \draw[->-=.5] (1,0) -- (.5,.5);
    \draw[style={dash pattern=on 0pt off 1,
line cap=round},->-=.5] (.5,.5) -- (.5,1);
    \filldraw (.5,.5) circle (1 pt);
    \node at (.5,0.1) {\scriptsize$\dots$};
    \end{tikzpicture}~,
    \quad\quad
    \begin{tikzpicture}[scale=1.5, baseline={([yshift=-.5ex]current bounding box.center)}]
    \draw[->-=.5] (.5,.5) -- (0,1);
    \draw[->-=.5] (.5,.5) -- (1,1);
    \draw[style={dash pattern=on 0pt off 1,
line cap=round},->-=.5] (.5,0) -- (.5,.5);
    \filldraw (.5,.5) circle (1 pt);
    \node at (.5,0.9) {\scriptsize$\dots$};
    \end{tikzpicture} ~.
\end{equation}
\end{definition}

We will use the following conventions on quantum integers: for $n \geq 0$, we write
$$[n] = \frac{q^n - q^{-n}}{q - q^{-1}}\, \text{ and } [n]! := [n] [n-1] \cdots [1].$$
\begin{prop} \label{prop:AIII-relations}
    The defect skein module $\sk_{(\GL_n, \GL_a\times \GL_b)}(M,K)$ is spanned by basic $(\GL_n,\GL_a\times \GL_b)$-skeins modulo the following skein relations: 
    \begin{gather} 
    \overcross{strand}{strand}
  \;-\;
  \undercross{strand}{strand}
  \;=\; \left(q-q^{-1}\right)\;
  \parallelstrands{strand,oriented}{strand,oriented}, \qquad  q^{-n}\overcross{dstrand}{dstrand}
  \;=\; \parallelstrands{dstrand,oriented}{dstrand,oriented} \;=\;
  q^{n}\undercross{dstrand}{dstrand}, \notag \\  \notag
  \circlestrand{strand}{ccw} = [n] \, \varnothing, \qquad \loopstrandJ{strand} = q^n \singlestrand{strand,oriented}, \qquad \circlestrand{dstrand}{ccw} =  \varnothing, \qquad  \loopstrandJ{dstrand} = q^{n} \singlestrand{dstrand,oriented},\\ \overcross{strand}{dstrand}
   \;=\;
  q^{2}\undercross{strand}{dstrand}, \qquad \overcross{dstrand}{strand}
 \;=\;
  q^{2}\undercross{dstrand}{strand},\label{eq:N-skein-relations}
\\ \notag
  \NtoNvertex{N} = \quad \frac{q^{\binom{n}{2}}}{[n]!} \sum_{\sigma \in S_n} (-q)^{-\ell(\sigma)} \, \Nbox{\sigma}, \notag\\ 
  \brmodstrands\; - \;\invbrmodstrands \;=\; (s-s^{-1})\;\parallelstrands{strand,oriented}{redstrand}\;, \qquad \detbrmodstrands ~=~ \detq(J^s) \parallelstrands{dstrand,oriented}{redstrand}~, \label{eq:N=p+q-skein-relations}\\ \notag
  \brmodtrace = \trq(J^s) ~ \singlestrand{redstrand} 
\end{gather}
\end{prop}
\begin{proof}
    Let $\operatorname{Web}_n$ denote the ribbon category consisting of objects which are oriented points on the negative $x$-axis of the 2-disk, and whose morphisms are oriented embedded tangles modulo the above given relations involving only black strands. It is known that the Cauchy completion $\widehat{\mathrm{Web}}_n$ for $q$ generic is equivalent to $\Rep_q(\GL_n)$ (see \cite{CKM}).
    
    We repeat this argument for the associated braided module category as follows. Let $\operatorname{Web}_{n=a+b}$ be the category whose objects are oriented points on the negative $x$-axis of the 2-disk represented as ordered tuples $(\epsilon_1,\dots, \epsilon_m)$ of signs $\epsilon_i\in \{\pm\}$ and its morphisms are oriented tangles embedded $\mathbb{D}\setminus\{0\}$ modulo the above given relations. This inherits the structure of a braided module category over $\operatorname{Web}_{n}$ and we claim that its Cauchy completion $\widehat{\operatorname{Web}}_{n=a+b}$ is equivalent to $\Rep_{q,s}(\GL_a \times \GL_b)$.  

    Indeed, as in the non-defect case we have a functor 
    \[F:\widehat{\operatorname{Web}}_{n=a+b}\to \Rep_{q,s}(\GL_a\times \GL_b)\] determined by sending $(\epsilon_1,\dots, \epsilon_m)\in \operatorname{Web}_{n=a+b}$ to $(V^{\epsilon_1}\otimes \cdots \otimes V^{\epsilon_m})\lact \mathbf{1}$ where $V$ denotes the fundamental $\GL_n$-representation. 
    Essential surjectivity of $F$ follows from the fact that the trivial representation is a $\Rep_q(\GL_n)$-generator and $V$ is a $\otimes$-generator for $\Rep_q(\GL_n)$. 
    For full faithfulness consider first the fundamental representation $V$, which decomposes into simple objects $V\cong \C^a\oplus \C^b$ as an object in $\Rep_{q,s}(\GL_a \times \GL_b)$. The endomorphism algebra $\End_{\GL_a\times \GL_b}(V)$ is 2-dimensional and it is easy to see that $\id$ and $J^s$ form a basis of it. The map \[F:\End_{\operatorname{Web}_{n=a+b}}(+)\to \End_{\GL_a\times \GL_b}(V)\] is bijective. Surjectivity is immediate from the two basis elements and injectivity follows from showing that the endomorphism ring $\End_{\operatorname{Web}_{n=a+b}}(+)$ is 2-dimensional and by resolving any ribbon tangle using the relations. 
\end{proof}

The first equation in \eqref{eq:N=p+q-skein-relations} is the Hecke relation $J^s\sim s$. The latter relations\footnote{Here $\detq(J^s)$ resp.\ $\trq(J^s)$ denote the evaluation of the character $\chi_{J^s}: \Oq(\GL_n)\rightarrow \C[q^\pm, s^\pm], a^i_j \mapsto J^i_j$ to the elements $\detq(A)$ resp.\ $\trq(A)$ in $\Oq(\GL_n)$.} involving $\detq(J^s)$ and $\trq(J^s)$ further depends on the type $(a,b)$ which are determined as follows: 
\begin{lem}
Let $n = a+b$ with $a \leq b$. We have
\begin{equation*}
\detq(J^s)=\begin{cases}
     (-1)^a q^{n}& \text{if }a = b\\
      (-1)^{b} q^{2ab}s^{-(n-2a)} & \text{if }a < b \\
    \end{cases} 
\end{equation*}
and
\begin{equation*}
\trq(J^s)=\begin{cases}
     (s-s^{-1})q^a[a]& \text{if }a=b\\
     (s-s^{-1})q^a[b] - s[n-2a] & \text{if }a < b \\
    \end{cases} 
\end{equation*}~.
\end{lem}
\begin{proof}
We need to evaluate the $\Oq(\GL_n)$-character $\chi_{J^s}(a^i_j):= (J^s)^i_j$ on the elements $\detq(A)$ and $\trq(A)$. These are given by:
\begin{equation*}
\detq(A):=  \sum_{\sigma\in S_n}{(-q)^{l(\sigma)}q^{e(\sigma)}a^1_{\sigma(1)}\cdots a^n_{\sigma(n)}} ~, 
\end{equation*}
where $l(\sigma)$ resp.\ $e(\sigma)$ is the length resp.\ exceedance\footnote{The length $l(\sigma)$ is the number of pairs $i<j$ in $\{1,\cdots,N\}$ such that $\sigma(j)>\sigma(i)$ and the exceedance $e(\sigma)$ is the number of $i\in\{1,\dots, N\}$ such that $\sigma(i)>i$.} of $\sigma$ \cite{JordanWhite},
and
\begin{equation*}
\trq(A) := \sum_{1\leq k\leq n}{q^{n+1-2k}a^k_k}~. 
\end{equation*}
The result follows from using the defining equation \eqref{eq:Js} of $J^s$ through straightforward computations.
\end{proof}

\newcommand{\arxiv}[2]{\href{http://arXiv.org/abs/#1}{#2}}
\newcommand{\doi}[2]{\href{http://doi.org/#1}{#2}}

\end{document}